\RequirePackage{etoolbox}
\csdef{input@path}{{style/}{graphics/}}
\documentclass[aap,MSNbibl,dvips]{arximspdf}
\makeatletter
   \@ifpackageloaded{graphicx}{}{\usepackage{graphicx}}
\makeatother

\doi{10.1214/14-AAP1031} 
\volume{25}
\issue{3}
\pubyear{2015}
\firstpage{1581}
\lastpage{1615}

\makeatletter
\newcommand{\rrVert}{\Vert}
\newcommand{\llVert}{\Vert}
\newproclaim{defn}{Definition}
\newproclaim{remark}{Remark}
\newtheorem{teo}{Theorem}[section]
\newtheorem{lem}[teo]{Lemma}
\newtheorem{prop}[teo]{Proposition}
\newtheorem{cor}[teo]{Corollary}
\newtheorem{question}{Question}
\makeatother

\begin{document}
\begin{frontmatter}

\title{Mixing time of Metropolis chain based on random transposition
walk converging to multivariate Ewens distribution\thanksref{T1}}
\runtitle{Mixing time of Ewens Metropolis random transposition walk\hspace*{10pt}}

\begin{aug}
\author[A]{\fnms{Yunjiang}~\snm{Jiang}\corref{}\ead[label=e1]{jyj@math.stanford.edu}}
\runauthor{Y. Jiang}
\affiliation{Stanford University}
\address[A]{Department of Mathematics\\
Stanford University\\
Stanford, California 94305\\
USA\\
\printead{e1}} 
\end{aug}
\thankstext{T1}{Supported in part by NSF Graduate Fellowship.}

\received{\smonth{5} \syear{2013}}
\revised{\smonth{3} \syear{2014}}

%
\begin{abstract}
We prove sharp rates of convergence to the Ewens equilibrium
distribution for a family
of Metropolis algorithms based on the random transposition shuffle on
the symmetric group,
with starting point at the identity. The proofs rely heavily on the
theory of
symmetric Jack polynomials, developed initially by
Jack [\textit{Proc. Roy. Soc. Edinburgh Sect.~A}  \textbf{69} (1970/1971) 1--18],
Macdonald [\textit{Symmetric Functions and Hall Polynomials} (1995) New York] and
Stanley [\textit{Adv. Math.} \textbf{77} (1989) 76--115].
This completes the analysis started by
Diaconis and Hanlon in [\textit{Contemp. Math.} \textbf{138} (1992) 99--117].
In the end we also explore other integrable Markov chains that can be
obtained from symmetric function theory.
\end{abstract}

%
\end{frontmatter}

\section{Introduction}
There is a well-known bijection between the set of partitions of $n$
and the conjugacy classes of the symmetric group $S_n$. The partition
that a permutation $\sigma\in S_n$ corresponds to is simply given by
its cycle structure. In fact this connection is the basis for the
classical representation theory of $S_n$ (see, e.g., \cite{FH}): the
set of irreducible representations of $S_n$ is indexed by the set
$\mathcal{P}_n$ of partitions of~$n$. Since $S_n$ is finite, it can
also be endowed with a probability space structure. The most natural
measure on $S_n$ is thus the uniform measure, with each permutation
getting a weight of $1/n!$. Sampling from this uniform measure is
important for many statistical applications \cite{PD}, such as
testing independence of $n$ i.i.d. uniform random variables on an
ordered set. Its intimate connection with card shuffling models has
also generated a wonderful array of mathematics, most notably the
determination of their mixing times; see, for instance, \cite{dovetail,DiaconisShahshahani} and \cite{MorrisLreversal}.

One of the most natural generalizations of the uniform measure on $S_n$
is a \mbox{1-}parameter family of so-called multivariate Ewens distributions,
named after Warren Ewens, who derived the partition function of this
probability measure. It is defined by giving each permutation $\sigma$
a weight of $\alpha^{\ell(\sigma)}$, $\alpha> 0$, where $\ell
(\sigma)$ is the number of cycles in $\sigma$; hence it can be viewed
as an exponentially tilted family based on the uniform measure. The
distribution was first applied to population genetics, in which it
describes the distribution of frequencies of alleles in a sample of
genes (\cite{Johnson}, Chapter~41).

Many important properties of the uniform measure on $S_n$ continues to
hold for the multivariate Ewens distribution \cite{EwensPoisson}. For
instance, the two perfect sampling schemes of the uniform measure,
Feller coupling and the Chinese restaurant process, both generalize to
the Ewens case. In this article, we describe a much more subtle
property of the uniform measure that has been successfully generalized
to the $\alpha$ deformed setting. In a nutshell, the characters of the
symmetric group $S_n$ form a basis in the Fourier space of class
functions on $S_n$ under the uniform measure \cite{PD}. When the
underlying measure is $\alpha$-deformed from the uniform, one can only
make sense of Fourier transforms of a particular type of class
functions, namely the ones supported on transpositions and the identity
class. In that case, the basis in the Fourier space becomes the matrix
coefficients of the transition from Jack symmetric polynomials basis to
the power sum symmetric polynomials basis. These generalize the
characters of $S_n$, which happen to be the transition coefficients
from Schur polynomial basis to the power-sum polynomial basis.

This Fourier analytic property was first obtained by Stanley in \cite
{StanleyJack}. Later Hanlon \cite{Hanlon} applied it to the study of
the Metropolis Markov chain based on random transposition walk on $S_n$
that converges to the multivariate Ewens distribution. Diaconis and
Hanlon \cite{DiaconisHanlon} further initiated the investigation of
total variation mixing time of this chain.

In light of the sharp result in \cite{DiaconisShahshahani} for the
uniform case ($\alpha=1$), it is natural to wonder what's the exact
mixing time for the Diaconis--Hanlon Metropolis walk ($\alpha\neq1$).
In this paper, we prove a pair of matching upper and lower bound for
the mixing time that applies to all $\alpha> 0$, which together imply
the cut-off phenomenon. Previously Diaconis and Hanlon \cite
{DiaconisHanlon} outlined a proof of the upper bound in the case
$\alpha> 1$ and conjectured that it was tight.\looseness=1

In the \hyperref[app]{Appendices}, we include some preliminary attempts to generalize
the walk studied here in various directions. These were motivated by
questions of Diaconis on whether other nontrivial Markov chains can be
constructed from symmetric function theory. First we look at the action
of the Sekiguchi--Debiard operator on other classical bases of
symmetric polynomials. We also consider higher order operators, as
given by the operator valued generating\vadjust{\goodbreak} function in \cite{Macdonald},
page 317. These turn out to give new local move Markov chains
converging to $\operatorname{MED}(\theta)$, albeit without simple group
theoretic interpretations. Finally we look at Laplace--Beltrami
operators associated with other root systems. Recall the Schur--Weyl
duality between the simple Lie groups $SU(n)$ and the finite groups
$S_n$. This leads to an interpretation of the Sekiguchi--Debiard
operator (as well as their quantized version given by Macdonald \cite
{Macdonald}) as associated with root system of type $A_n$. The
appropriate generalizations were first discovered by Heckman and Opdam
\cite{HeckmanOpdamI} in the context of Hamiltonian systems of
particles on a circle, and later extended to the Macdonald case in
\cite{MacdonaldRoots}; see also \cite{Koornwinder} and references
therein for a 5-parameter generalization.

Jack polynomials, which form the backbone of the argument presented
here, turn out to be special cases of Macdonald polynomials of type
$A_n$. Diaconis and Ram \cite{Aux} interpreted them as eigenfunctions
of an auxilliary variable algorithm on the space of partitions, which
can be viewed as a quantized version of the local walk studied here.

\section{Metropolis walk starting at the identity class}

The multivariate Ewens distribution with parameter $\alpha$ is defined
on $S_n$ with $P(\sigma)$ proportional to $\alpha^{\ell(\sigma)}$
where $\ell(\sigma)$ is the number of cycles of $\sigma$. The
normalization constant $z_n(\alpha) = \alpha_{(n)}:= \alpha(\alpha
+1) \cdots(\alpha+ n-1)$.

Consider now the random transposition walk on $S_n$, defined by picking
a pair of numbers $i\neq j$ at random, and multiplying the current
state in $S_n$ by the transposition $(ij)$. In this form, the walk is
periodic and does not converge. But if we make it lazy, then it
converges to the uniform measure on $S_n$. By metropolizing the
nonlazy walk to the multivariate Ewens distribution with parameter
$\alpha$, we create a new Markov chain that converges to $\operatorname{MED}(\alpha)$; see \cite{DiaconisHanlon} for details on the
Metropolis algorithm. The walk behaves differently for $\alpha> 1$ and
$\alpha< 1$. And as long as $\alpha\neq1$, it converges, because it
always has positive holding probability.

\emph{Warning}. The $\alpha$ parameter here will be the reciprocal of
$\theta$ below.

If we start the Metropolis walk described in the introduction from the
identity element, then we can view it as a walk either on the symmetric
group or on the set of partitions. It is the latter interpretation that
allows for sharp analysis with other starting points.

%
\begin{teo} \label{Main}
For $\theta\in(0, \infty) \setminus\{1\}$, let $P_\theta$ be the
discrete time Metropolis chain based on random transposition walk
starting\vspace*{1pt} from $\operatorname{id}$, converging to the multivariate
Ewens distribution $\pi$ with parameter $\theta^{-1}$ (so identity
has the largest mass when $\theta< 1$). Explicitly, let $\lambda(\pi
)$ be the cycle structure of the permutation $\pi$, $\lambda^t$ be
the transposition of $\lambda$ as a Ferrers diagram and denote
$n(\lambda) = \sum_{i=1}^n {\lambda^t_i\choose 2}$. Then the
transition rule is given by
\begin{eqnarray*}
P_\theta(\pi,\sigma) &=& \cases{ \displaystyle1- 1 \wedge\theta+
\frac{n(\lambda(\pi)^t)}{{n\choose 2}} \bigl(1 \wedge\theta- 1 \wedge\theta^{-1}\bigr),
\vspace*{3pt}\cr
\hspace*{84.5pt}\mbox{if $\sigma =\pi$,}
\vspace*{3pt}\cr
\displaystyle\frac{1}{{n\choose 2}} (1 \wedge\theta),\hspace*{33pt}\mbox{if $\sigma= \pi(i,j)$ and $\ell(\sigma) = \ell(\pi)-1$},
\vspace*{5pt}\cr
\displaystyle
\frac{1}{{n\choose 2}}\bigl(1 \wedge\theta^{-1}\bigr),\qquad\mbox{if $\sigma= \pi(i,j)$ and $\ell(\sigma) = \ell(\pi) +1$},
\vspace*{5pt}\cr
0,\hspace*{74pt}\mbox{otherwise.}}
\end{eqnarray*}
Here $1 \le i < j \le n$.
The chain $P_\theta$ has a total variation cut-off at $t = \frac{1}{2}
(\frac{1}{\theta} \vee1) n \log n$. This means
\begin{eqnarray*}
\lim_{c \to\infty}\limsup_{n \to\infty} \bigl\llVert
P_{\operatorname
{id}}^{t(c)} - \pi\bigr\rrVert _{\mathrm{TV}} &=& 0,
\\
\lim_{c \to-\infty} \liminf_n \bigl\llVert
P_{\operatorname{id}}^{t(c)} - \pi\bigr\rrVert _{\mathrm{TV}} &=& 1
\end{eqnarray*}
for $t(c):= \frac{1}{2} (\frac{1}{\theta} \vee1) n (\log n + c)$.
\end{teo}

%
\begin{remark}
(1)~Notice the chain has no intrinsic holding: when $\theta= 1$ it
corresponds to the completely industrious random transposition walk
with probability~$\frac{1}{{n\choose 2}}$ to go to a neighboring
permutation. If one inserts a holding of $1/n$, that is, $P \mapsto
\frac{1}{n}I + (1 - \frac{1}{n}) P$, then the asymptotic cut-off
profile stays the same and has no removable discontinuity at $\theta=1$.

(2) The Metropolis chain defined above can be projected to conjugacy
classes of $S_n$, namely partitions, provided we start at the identity
element. The transition matrix takes the following form:
\begin{eqnarray*}
P_\theta(\lambda,\mu) &=& \cases{ \displaystyle1- 1 \wedge\theta+
\frac{n(\lambda^t)}{{n\choose 2}} \bigl(1 \wedge\theta- 1 \wedge\theta^{-1}\bigr),
\vspace*{3pt}\cr
\hspace*{90pt}\mbox{if $\mu=\lambda $,}
\vspace*{3pt}\cr
\displaystyle\frac{\lambda_i \lambda_j}{{n\choose 2}} (1 \wedge\theta),\hspace*{30pt}\mbox{if $\mu_k = \lambda_i + \lambda_j$,}
\vspace*{5pt}\cr
\displaystyle\frac{\lambda_k}{{n\choose 2}}\bigl(1 \wedge\theta ^{-1}\bigr),\hspace*{27.5pt}\mbox{if $\mu_i + \mu_j = \lambda_k$ and $\mu_i \neq\mu _j$,}
\vspace*{5pt}\cr
\displaystyle\frac{\lambda_k}{2 {n\choose 2}}\bigl(1 \wedge\theta ^{-1}\bigr),\qquad\mbox{if $\mu_i +\mu_j = \lambda_k$ and $\mu_i = \mu_j$,}
\vspace*{3pt}\cr
0,\hspace*{79pt}\mbox{otherwise.}}
\end{eqnarray*}

Here $1 \le i < j \le\ell(\mu)$ and $1 \le k \le\ell(\lambda)$.
Furthermore in the second line, $\mu\setminus\mu_k = \lambda
\setminus\{\lambda_i, \lambda_j\}$. In other words, $\mu$ is
obtained from $\lambda$ by joining $\lambda_i$ and $\lambda_j$ into
a single part $\mu_k$. Similarly, for the third and fourth lines of
the formula above, $\mu\setminus\{\mu_i, \mu_j\} = \lambda
\setminus\lambda_k$, that is, $\mu$ is obtained by breaking a part
$\lambda_k$ in $\lambda$ into two parts, $\mu_i$ and $\mu_j$.

(3) The first order phase transition at $\theta=1$ for the cut-off
value is not surprising, because the Metropolis chain has different
forms for $\theta< 1$ and for $\theta> 1$.

(4) $\theta$ denotes the inverse of the Ewens sampling parameter the
chain $P_\theta$ converges to. This choice of convention is justified
by the fact that the left eigenfunctions of the chain $P_\theta$ are
the transition coefficients from the Jack polynomials with parameter
$\theta$ to the power sum polynomials, as derived in~\cite{Hanlon}.

(5) It will be interesting to see what happens when the $\theta$
value in the transition probability is allowed to be state dependent,
but satisfying some uniform bound $c < \theta(\lambda) < c^{-1}$ for\vspace*{1pt}
$c$ independent of $n$. My conjecture is that it will always take at~least $\frac{1}{2} n \log n$ steps\vspace*{1pt} to converge to its stationarity
distribution, which is no longer in the Ewens family.
\end{remark}

The next four sections will be devoted to the proof of Theorem~\ref{Main}.

\section{Preliminaries on $\mathcal{L}^2$ mixing time}

%
\begin{lem} \label{doublystochastic}
Given a \emph{reversible} ergodic Markov chain $P$ on a finite state
space $X$,
let $f_j$ be the right eigenfunctions, normalized so that
\[
\sum_x f_j(x)^2 \pi(x)
= 1,
\]
with corresponding eigenvalues $\beta_j$. Then $g_j(x):= f_j(x) \pi
(x)$ are left eigenfunctions of $P$, with the \emph{same} eigenvalues,
satisfying
\[
\sum_x g_j(x)^2
\frac{1}{\pi(x)} =1.
\]
Furthermore,
%
\begin{eqnarray}
\frac{1}{\pi(x)} &=& \sum_j f_j^2(x),
\label{rightsum}
\\
\pi(x) &=& \sum_j g_j^2(x).
\label{leftsum}
\end{eqnarray}
\end{lem}

\begin{pf}
By reversibility, we have $\pi(x) P(x,y) = \pi(y) P(y,x)$. Therefore,
\begin{eqnarray*}
\beta_j f_j(x) &=& \sum_y
P(x,y) f_j(y)=\sum_y
\frac{\pi(x)}{\pi
(x)} P(x,y) f_j(y)
\\
&=& \sum_y
\frac{\pi(y)}{\pi(x)} P(y,x) f_j(y).
\end{eqnarray*}
Now multiplying both sides by $\pi(x)$, we get
\[
\beta_j \pi(x) f_j(x) = \sum
_y \pi(y) f_j(y) P(y,x).
\]
This proves the first part.

The last two identities are nothing but a restatement of the fact that
the matrix $\Pi(x,y):= \pi(x) P(x,y) /\pi(y)$ is doubly stochastic;
that is, each row and column sums to $1$. Here is a formal proof. Since
$\{f_j\}$ forms a basis, we can decompose the function $z \mapsto
1_x(z)$ in it,
\[
1_x(z) = \sum_j c_j
f_j,
\]
where the coefficients $c_j$ are given by
\[
c_j = \langle f_j, 1_x
\rangle_{\mathcal{L}^2(\pi)}= \sum_z
1_x(z) f_j(z) \pi(z) = f_j(x) \pi(x).
\]
The first equality follows immediately. The second is similar.
\end{pf}

%
\begin{lem}
Under the same notation as the previous lemma, one can bound the total
variation distance to stationarity at time $k$ starting at state $x$ by
%
\begin{eqnarray}
4\bigl\llVert P_x^k - \pi\bigr\rrVert
_\mathrm{TV}^2 &\le&\biggl\llVert \frac
{P_x^k}{\pi(x)} - 1\biggr
\rrVert _{\mathcal{L}^2(\pi)}^2 \label
{Cauchy-Schwarz}
\\
&=& \frac{1}{\pi^2(x)} \sum_j
\beta_j^{2k} g_j^2(x) -1.
\label
{upperboundlemma}
\end{eqnarray}
\end{lem}

\begin{pf}
The first inequality (\ref{Cauchy-Schwarz}) follows directly from the
Cauchy--Schwarz inequality. To prove the second formula (\ref{upperboundlemma}), first write
\begin{eqnarray*}
\biggl\llVert \frac{P_x^k}{\pi} - 1\biggr\rrVert _2^2
&=& \sum_y \pi(y)\biggl[\biggl(\frac
{P_x^k(y)}{\pi(y)}
\biggr)^2 -1\biggr]= \sum_y
\frac{(P_x^k(y))^2}{\pi(y)} - \pi(y).
\end{eqnarray*}
Using reversibility again [in the extended form $\pi(x) P^k(x,y) = \pi
(y) P^k(y,x)$], we can write
\[
\frac{(P_x^k(y))^2}{\pi(y)} = \frac{P_x^k(y) P_y^k(x)}{\pi(x)}.
\]
Thus summing over $y \in X$, we get
\begin{eqnarray*}
\biggl\llVert \frac{P_x^k}{\pi} -1 \biggr\rrVert _2^2 &=&
\frac{P^{2k}(x,x)}{\pi
(x)} - 1.
\end{eqnarray*}

Next write the function $y \mapsto P_x^{2k}(y)$ as the result of a row
vector multiplied by a matrix,
\begin{eqnarray*}
P_x^{2k}(y) &=& \sum_z
1_x(z) P^{2k}(z,y).
\end{eqnarray*}
By the previous lemma, we have
%
\begin{eqnarray}
\label{indicator} 1_x(y) &=& \sum_j
c_j g_j(y),
\end{eqnarray}
where $c_j = \sum_z \frac{1}{\pi(z)} 1_x(z) g_j(z) = \frac
{g_j(x)}{\pi(x)}$.\vspace*{2pt}

Finally evaluating at $y = x$ in (\ref{indicator}), we obtain
\begin{eqnarray*}
\frac{P^{2k}(x,x)}{\pi(x)} &=& \sum_j \frac{g_j(x)^2}{\pi(x)^2}
\beta_j^{2k}.
\end{eqnarray*}\upqed
\end{pf}

\section{Results from symmetric function theory}

First we recall from Hanlon~\cite{Hanlon} that the eigenvalues for the
chain $P_\theta$ projected onto conjugacy classes of $S_n$, with
$\theta> 1$, are given by
%
\begin{equation}
\label{Hanlonbeta} \beta_\lambda= \frac{n(\lambda^t) - \theta^{-1} n(\lambda
)}{{n\choose 2}}.
\end{equation}
For an independent proof with pointers to literature, see the proof of
Theorem~\ref{diagonalization}.

Next we derive the eigenvalues of the chain $P_{\theta}$, for $\theta
\in(0,1)$. Notice this is not the same chain as that studied in \cite
{Hanlon}. Here the identity element gets the biggest mass, whereas in
\cite{Hanlon}, identity has the smallest mass [Ewens sampling with
parameter $ \in(0,1)$]. But the same result of Macdonald can be used
here to derive eigenvalues. Indeed, consider the following matrix
$T_\theta$ defined by
\begin{eqnarray*}
T_\theta(\pi,\sigma) &=& \cases{ \displaystyle\frac{(\theta-1 )n(\pi)}{\theta{n\choose 2}}, &\quad
if $\sigma=\pi$,
\vspace*{3pt}\cr
\displaystyle\frac{1}{{n\choose 2}}, &\quad if $\sigma= \pi
(i,j)$ and $\ell(\sigma) = \ell(\pi)-1$,
\vspace*{5pt}\cr
\displaystyle\frac{1}{ \theta{n\choose 2}}, &
\quad if $\sigma= \pi(i,j)$ and $\ell(\sigma) = \ell(\pi) +1$,
\vspace*{5pt}\cr
0, &\quad
otherwise,}
\end{eqnarray*}
where $n(\pi) = \sum_i {\pi^t_i\choose 2}$ and $\{\pi_i\}$ is
the partition structure of $\pi$. This quantity gives the number of
ways to break a part in the partition structure of $\pi$ into two
parts, using multiplication by a transposition.

Hanlon considered the case $\theta\ge1$ (his $\alpha$ is our $\theta
$), here we extend to $\theta\in(0,1)$, which is no longer a Markov
matrix because the diagonal entries are no longer nonnegative.
Nevertheless The rows still sum to $1$. Then his Theorem 3.5 continues
to hold because the proof never uses $\theta\ge1$. Likewise, Theorem
3.9 holds for $\theta< 1$. To get $P_\theta$, we simply need to
rescale $T_\theta$ by $\theta$ and add a constant multiple $cI$ of
identity matrix.\vspace*{1pt} $c$ can be obtained by looking at the top eigenvalue.
By Theorems 3.5~and~3.9 of \cite{Hanlon}, $\theta T_\theta$ has
eigenvalues $\frac{\theta n(\lambda^t) - n(\lambda)}{{n\choose
2}}$. Thus we need to add $1-\theta$ in order for $\beta_{(n)}$ to
equal~$1$.

Combining the two cases, we have the following formula for eigenvalues
of $P_\theta$:
%
\begin{equation}
\label{betalambda} \beta_\lambda(\theta) = 1 - \theta\wedge1 +
\frac{ \theta
n(\lambda^t) - n(\lambda)}{(\theta\vee1) {n\choose 2}}.
\end{equation}

Denote $r(\lambda) = \frac{n(\lambda^t) - n(\lambda)}{{n\choose 2}}$. The following lemma collects a bunch of estimates about
$\beta_\lambda$:

%
\begin{lem} \label{technical}
Let $\succeq$ be the natural partial order on the set of partitions
defined as follows: given two partitions represented by Ferrers
diagrams $\lambda$ and $\lambda'$, say $\lambda\succeq\lambda'$ if
$\lambda$ can be obtained by successive up and right moves of blocks
of $\lambda'$.

\begin{longlist}[(3)]
\item[(1)] $n(\lambda^t)$ is monotone, and $n(\lambda)$ is anti-monotone
in the above partial order; that is, for $\lambda\succeq\lambda'$,
\begin{eqnarray*}
n\bigl(\lambda^t\bigr) &>& n\bigl(\lambda'^t
\bigr),
\qquad
n(\lambda) < n\bigl(\lambda'\bigr).
\end{eqnarray*}

\item[(2)]
$\beta_\lambda$ is monotone with respect to the natural partial order
on $\lambda$. Thus \mbox{$\beta_{(n)} \ge\beta_\lambda$} for all $\lambda
\vdash n$, $\beta_\lambda\le\beta_{(\lambda_1, n-\lambda_1)}$ and
$\beta_\lambda\ge\beta_{(\lambda_1, 1^{n-\lambda_1})}$.

\item[(3)]
Furthermore, for $\lambda_1 \ge\frac{n}{2}$,
%
\begin{equation}
\label{biglambda1} \beta_\lambda\le1 - (\theta\wedge1) \frac{2\lambda_1(n-\lambda
_1)}{n(n-1)},
\end{equation}
and in general
%
\begin{equation}
\label{generallambda} \beta_\lambda\le1 - (\theta\wedge1) \biggl(1-
\frac{\lambda_1 -1}{n-1}\biggr).
\end{equation}
In particular, if $\beta_\lambda(\theta) \ge0$, the above two
inequalities hold with $|\beta_\lambda|$.

\item[(4)]
Finally, if $\beta_\lambda(\theta) < 0$, then $\beta_{\lambda
^t}(\theta) > 0$ and $|\beta_\lambda| \le\beta_{\lambda^t}$.
\end{longlist}
\end{lem}

\begin{pf}
(1) It suffices to check the first assertion for $\lambda$ and
$\lambda'$ that differ by one block, that is, $\lambda_i = \lambda
_i'+1$, $\lambda_j = \lambda_j'-1$, $i< j$. Then
\begin{eqnarray*}
n\bigl(\lambda^t\bigr) - n\bigl(\lambda'^t
\bigr) &=& \tfrac{1}{2} \bigl[\lambda_i(\lambda _i-1)
+ \lambda_j(\lambda_j-1) - \lambda'_i
\bigl(\lambda'_i-1\bigr) - \lambda '_j
\bigl(\lambda'_j-1\bigr) \bigr]
\\
&=& \lambda_i' - \lambda_j \ge0,
\end{eqnarray*}
using the fact $\lambda_i \ge\lambda_j$ and $\lambda_i' \ge\lambda
_j'$ by definition of Ferrers diagram.
The antimonotonicity of $n(\lambda)$ follows by taking transpose.

(2) This follows directly from the previous assertion and formula
(\ref{betalambda}) for $\beta_\lambda$ in terms of $n(\lambda)$
and $n(\lambda^t)$.

(3) Equation~(\ref{biglambda1}) follows from $\lambda\preceq(\lambda_1,
n-\lambda_1)$ and monotonicity, after throwing away the term $-\frac
{n(\lambda)}{(\theta\vee1) {n\choose 2}}$.

For (\ref{generallambda}), we again throw away the term $-\frac
{n(\lambda)}{{n\choose 2}}$ in $\beta_\lambda$ to obtain
\begin{eqnarray*}
\beta_\lambda&\le&1 - \theta\wedge1 + \frac{\theta n(\lambda
^t)}{(\theta\vee1) {n\choose 2}}= 1-(\theta
\wedge1) \biggl(1-\frac
{n(\lambda^t)}{{n\choose 2}} \biggr)
\\
&=& 1- (\theta\wedge1) \biggl(1-
\frac
{\sum_j \lambda_j(\lambda_j -1)}{n(n-1)} \biggr)
\\
&\le&1 - (\theta\wedge1) \biggl(1-\frac{(\lambda_1 -1)\sum_j \lambda
_j}{n(n-1)}\biggr)= 1 -(\theta\wedge1)
\biggl(1-\frac{\lambda_1 -1}{n-1}\biggr).
\end{eqnarray*}

(4) Here we consider $\theta\ge1$ and $\theta< 1$ separately.
When $\theta\ge1$,
\begin{eqnarray*}
\beta_\lambda&=& \frac{1}{\theta{n\choose 2}} \bigl(\theta n\bigl(\lambda
^t\bigr) - n(\lambda)\bigr).
\end{eqnarray*}
If $\beta_\lambda\le0$, then $\theta n(\lambda^t) - n(\lambda) \le
0$. Since $\theta\ge1$, $n(\lambda) \ge n(\lambda^t)$. So
\begin{eqnarray*}
\theta n(\lambda) - n\bigl(\lambda^t\bigr) \ge n(\lambda) - \theta n
\bigl(\lambda ^t\bigr) &\ge&0,
\end{eqnarray*}
which implies $\beta_{\lambda^t} \ge|\beta_\lambda| \ge0$.

Next let $\theta< 1$. Then we can write
\begin{eqnarray*}
\beta_\lambda&=& 1- \theta+ \theta\frac{n(\lambda^t)}{{n\choose 2}} -
\frac{n(\lambda)}{{n\choose 2}}= \biggl[1- \frac
{n(\lambda)}{{n\choose 2}}\biggr] - \theta\biggl[ 1-
\frac{n(\lambda
^t)}{{n\choose 2}}\biggr].
\end{eqnarray*}
If $\beta_\lambda\le0$, then since $\theta< 1$,
\[
1- \frac{n(\lambda^t)}{{n\choose 2}} \ge1- \frac{n(\lambda
)}{{n\choose 2}}.
\]
Switching $\lambda$ and $\lambda^t$ we again get
\[
\beta_{\lambda^t} \ge|\beta_\lambda| \ge0.
\]\upqed
\end{pf}

We also need some definitions and results from Diaconis and Hanlon
\cite{DiaconisHanlon}:

\begin{defn} \label{definition1}
We collect some notation to be used in the main proof below, some of
which will be repeated; they are, for the most part, taken from \cite
{DiaconisHanlon}:
\begin{longlist}[(3)]
\item[(1)]
Given a partition $\lambda\vdash n$, and a position $s=(i,j)$ in its
Ferrers diagram (i.e., $j \le\lambda_i$), define
\begin{eqnarray*}
h^*(s) &=& h^*_\lambda(s):= (a+1) \theta+ \ell,
\\
h_*(s) &=& h_*^\lambda(s):= a \theta+ (\ell+1),
\end{eqnarray*}
where $a = \lambda_i -j$ denotes the number of positions in the same
row and strictly to the right of $s$ (the arm length), and $\ell=
\lambda^t_j - i$ denotes the number of positions in the same column
and strictly below $s$ (the leg length).
\item[(2)]
Define the generalization of hooklength product,
\[
j_\lambda= j_\lambda(\theta):= \prod
_{s \in\lambda} h_*(s) h^*(s).
\]
When $\theta= 1$, this becomes the product of the hooklengths of all
the blocks in the diagram of $\lambda$.
\item[(3)] Define $c_{\lambda, \rho} = c_{\lambda,\rho}(\theta)$ to be
the change of basis coefficients from Jack symmetric polynomials
$J_\lambda(\theta)$ (not to be confused with $j_\lambda$ above) to
power sum polynomials, that is,
%
\begin{eqnarray}
\label{Jackpower} J_\lambda(\theta) &=& \sum_{\rho\vdash n}
c_{\lambda,\rho}(\theta) p_\rho.
\end{eqnarray}
See \cite{StanleyJack} for extensive development of properties of Jack
polynomials. When \mbox{$\theta= 1$}, $J_\lambda(1) = H_\lambda s_\lambda$,
where $H_\lambda= j_\lambda(1)$ is the hooklength product, and
$s_\lambda$ is the Schur polynomial indexed by $\lambda$.

\item[(4)] Denote\vspace*{1pt} by $\pi= \pi_\theta$ the Ewens sampling measure with
parameter $\theta^{-1}$; recall $\pi_\theta(\sigma) = \theta
^{-\ell(\sigma)} / z_n(\theta^{-1})$, where $z_n(\theta^{-1})=
\prod_{i=1}^n (\theta^{-1} + i-1)$ is the\vspace*{2pt} Ewens sampling formula.
Also let $\Pi= \Pi_{n,\theta}:= \pi_\theta(1^n)^{-1} = \prod_{i=1}^n (1 + \theta(i-1))$.
\end{longlist}
\end{defn}

Note that when $\theta=1$, $j_\lambda(1)$ is exactly the square of
the product of hook lengths of all positions in $\lambda$, which is
well known to be $(\frac{n!}{\dim\pi_\lambda})^2$ by the\vspace*{1pt} hooklength
formula. By Wedderburn's structure theorem (see \cite{DummitFoote},
Chapter~18, Theorem~10), we also have
\[
n! = \sum_\lambda\dim\pi_\lambda^2,
\]
therefore $\sum_\lambda\frac{1}{j_\lambda(1)} = \frac{1}{n!}$.

%
\begin{teo}\label{diagonalization}
The left eigenfunctions of the Metropolis chain $P_\theta$ defined on
partitions, normalized in $\mathcal{L}^2(\mathcal{P}_n, 1/\pi_\theta
)$, are given by
%
\begin{eqnarray}
\label{lefteigenfunction} g_\lambda(\rho) &=& \frac{c_{\lambda,\rho}}{(j_\lambda\Pi/ (\theta
^n n!))^{1/2}}
\end{eqnarray}
with corresponding eigenvalues stated in (\ref{betalambda}).
\end{teo}

\begin{pf}
We synthesize the arguments found in \cite{Hanlon}, Definition~3.8 to
Definition~3.12 and \cite{DiaconisHanlon}, Theorem~1. The result from
\cite{Macdonald}, Chapter VI, Section~4, shows that the Macdonald
polynomials are simultaneous eigenfunctions of the Macdonald operators
$D_{q,t}^r$, $r=0,\ldots, n$. Specializing to the limit $q=t^\theta$,
$t \to1$ and after some affine linear transformation, the same results
hold for Jack polynomials and the associated Sekiguchi--Debiard
operators (\ref{Sekiguchi-Debiard}), $D_\theta(X)$. The $X^2$
coefficient of this operator valued generating function turns out to be
the following Laplace--Beltrami-type operator (our notation differs
slightly from \cite{Macdonald}, page 320): let $f$ be a homogeneous
polynomial of degree $N$ in $n$ variables, then
\[
D_\theta^2 f = \biggl(-\frac{\theta^2}{2} U_n -
\theta V_n + c_n\biggr) f,
\]
where $U_n = \sum_{i=1}^n (x_i\partial_i)^2 - x_i \partial_i = \sum_{i=1}^n x_i^2 \partial_i^2$, $V_n =\frac{1}{2} \sum_{i \neq j}
\frac{x_i^2 \partial_i - x_j^2 \partial_j}{x_i - x_j}$, and $
c_n = \theta^2 {N\choose 2} + \theta N {n\choose 2} + \frac
{1}{4} {n\choose 3} (3n-1)$; see (\ref{Dtheta^2}) for a proof.\vspace*{1pt}

After an affine transform, we arrive at the following cleaner operator:
%
\begin{equation}
\label{secondorderLaplace-Beltrami} L^2_\theta:= \frac{1}{{N\choose 2}}\biggl(
\frac{1}{2}U_n +\frac
{1}{\theta} \bigl(V_n -
(n-1)N \bigr)\biggr),
\end{equation}
which readily admits a Markov chain interpretation, when acting on
power sum polynomials. Combining (\ref{D0,0;2}), (\ref{D1,1;0})
and (\ref{Dtheta^2}), we have
\[
L_\theta^2 p_\lambda 
=
\frac{p_\lambda}{{N\choose 2}}\Biggl( \sum_{s < t}
\lambda_s \lambda_t \frac{p_{\lambda_s + \lambda_t}}{p_{\lambda_s}
p_{\lambda_t}} +\bigl(1-
\theta^{-1}\bigr) n\bigl(\lambda'\bigr) +
\theta^{-1} \sum_s \frac{\lambda_s}{2}
\sum_{r=1}^{\lambda_s -1} \frac{p_r p_{\lambda
_s -r}}{p_{\lambda_s}} \Biggr).
\]
Observe that for $\theta> 1$, the first and third terms above
correspond to joining two cycles into one and splitting a cycle into
two cycles, respectively, whereas the middle term gives the holding
probability at $\lambda$. In other words, the probability of going
from $\lambda$ to $\mu$ in one step under the Jack--Metropolis walk
is given by the $p_\mu$ coefficient of $L_\theta^2 p_\lambda$. This
translates to $L_\theta^2 p_\lambda= \sum_{\mu\vdash N} T_\theta
(\lambda,\mu) p_\mu$.

Next we show that $L_\theta^2 J_\lambda= \beta_\lambda J_\lambda$,
with $\beta$ given by (\ref{Hanlonbeta}) (i.e., when $\theta>1$);
the general case follows by an appropriate affine transform. In \cite
{Macdonald}, page 317, it is shown that for the Macdonald
operator-valued generating function $D_n(X;q,t)$, eigenvalues are given
by $\beta_\lambda(X;q,t) = \prod_{i=1}^n(1+X t^{n-i} q^{\lambda
_i})$. Now using Example~3(c) on page 320, one can derive the
eigenvalues for $D_n(X;\alpha)$, by considering the limiting operator
$\lim_{t\to1} (t-1)^{-n} Y^n D_n(Y^{-1};q,t)$ where $Y = (t-1)X -1$.
Extracting the $X^{n-2}$ term gives $\beta_\lambda$, which is stated
in\vspace*{1pt} Example~3(b) of page 327 as $\alpha{N\choose 2} e_\lambda
(\alpha)$ ($\alpha$ is the same as $\theta$ in our notation), since\vspace*{1pt}
$\Box_n^\alpha= \alpha{N\choose 2} L_\alpha^2$.\vadjust{\goodbreak}

Finally we prove the formula for the left eigenfunctions. Define the
inner product $\langle\cdot, \cdot\rangle_\theta$ by $\langle p_\lambda,
p_\mu\rangle_\theta= \delta_{\lambda\mu} z_\lambda\theta^{\ell
(\lambda)}$. In \cite{StanleyJack} (see also Lemma~3.11 of \cite
{Hanlon}), it is shown that $\langle J_\lambda, J_\lambda\rangle
_\theta= j_\lambda(\theta)$ as defined before. Here recall the
normalization of~$J_\lambda$ is fixed by requiring that in the
monomial symmetric function basis, its $m_{1^N}$ coefficient be $1$.
Therefore expressing $J_\lambda$ in terms of $p_\mu$'s, we have
\[
\sum_{\rho\vdash N} c_{\lambda,\rho}^2
z_\rho\theta^{\ell(\rho
)} = j_\lambda.
\]
On the other hand, the normalization constant for the $\operatorname{MED}(\theta^{-1})$ distribution is $z_n(\theta^{-1}) = \theta^{-1}
(\theta^{-1} + 1) \cdots(\theta^{-1} + n-1) = \Pi\theta^{-n}$,
hence $\pi_\theta(\rho) = \theta^{-\ell(\rho)} \frac{n!}{z_\rho
} z_n(\theta)^{-1}$, and with $g_\lambda(\rho)$ given in (\ref{lefteigenfunction}) we have
\[
\sum_{\rho\vdash N} g_\lambda(\rho)^2
\pi_\theta(\rho)^{-1} = \sum_\rho
\frac{c_{\lambda,\rho}^2}{j_\lambda\Pi/ (\theta^n n!)} \Pi\theta^{\ell(\rho) - n} \frac{z_\rho}{n!} = \sum
_\rho\frac
{c_{\lambda,\rho}^2 \theta^{\ell(\rho)} z_\rho}{j_\lambda} =1,
\]
by the previous equation. This shows $g_\lambda$ are indeed left
eigenfunctions by Lemma~\ref{doublystochastic}.
%
\end{pf}

%
\begin{cor} \label{righteigenfunctions}
The right eigenfunctions of $P_\theta$ are proportional to
\[
f_\lambda(\rho) = g_\lambda(\rho) \theta^{\ell(\rho)}
z_\rho.
\]
\end{cor}

\begin{pf}
Since $\pi_\theta(\rho) \propto\theta^{-\ell(\rho)} \frac
{n!}{z_\rho}$, this follows from Lemma~\ref{doublystochastic} and
the previous theorem.
\end{pf}

%
\begin{lem} \label{1^n}
For any $\lambda\vdash n$,
\[
c_{\lambda,1^n} = 1.
\]
\end{lem}

\begin{pf}
This follows from the following formula in \cite{StanleyJack}:
\[
J_\lambda\bigl(1^n; \theta\bigr) = \prod
_{(i,j) \in\lambda} \bigl(n -(i-1) + \theta(j-1)\bigr),
\]
true for all $n \in\mathbb{N}$, by reading coefficients of powers of
$n$; See \cite{DiaconisHanlon}, Section~4, Theorem~1.
\end{pf}

%
\begin{lem} \label{jlambdasplit}
$j_\lambda$ admits the following inductive bound on the parts of
$\lambda$:
\[
j_\lambda\ge\lambda_1!^2 \theta^{2 \lambda_1 -1}
\lambda _1^{\theta^{-1} - 1} e^{-\pi^2/ 12 \theta^2} j_{(\lambda_2, \ldots, \lambda_n)}.
\]
\end{lem}

Note that the constant $e^{-\pi^2/12 \theta^2}$ is not important.\vadjust{\goodbreak}

\begin{pf*}{Proof of Lemma \ref{jlambdasplit}}
From the definition, we have
%
\begin{eqnarray}\label{jlambdabound}
j_\lambda&\ge&\Biggl[\prod_{i=1}^{\lambda_1}
(i \theta) \prod_{i=1}^{\lambda_1 -1} (i\theta+ 1)
\Biggr] j_{(\lambda_2, \ldots, \lambda
_n)}\nonumber
\\
&=& \Biggl[\prod_{i=1}^{\lambda_1}
(i\theta) \prod_{i=1}^{\lambda_1
-1} (i \theta)
\bigl(1 + (i \theta)^{-1}\bigr)\Biggr] j_{(\lambda_2, \ldots,
\lambda_n)}
\nonumber\\[-8pt]\\[-8pt]
&\ge&\lambda_1 ! \theta^{\lambda_1} (\lambda_1 -1)!
\theta ^{\lambda_1 -1} \exp\biggl(\frac{1}{\theta} \log\lambda_1 -
\frac
{1}{\theta^2} \frac{\pi^2}{12}\biggr) j_{(\lambda_2, \ldots, \lambda
_n)}
\nonumber
\\
&=& \lambda_1 ! \theta^{\lambda_1} (\lambda_1 -1)!
\theta^{\lambda
_1 -1} \lambda_1^{1/\theta} \exp\biggl(-
\frac{1}{\theta^2} \frac{\pi
^2}{12}\biggr)j_{(\lambda_2, \ldots, \lambda_n)},\nonumber
\end{eqnarray}
where we used the fact that $1+x \ge e^{x - x^2/2}$ for $x \ge0$,
applied to $x = (i \theta)^{-1}$, and the zeta sum,
\[
\sum_i \frac{1}{2i^2} \le\frac{\pi^2}{12}.
\]\upqed
\end{pf*}

\section{Mixing time upper bound}
By Theorem~\ref{diagonalization}, under the same notation there, \emph
{four times} the total variation distance of $P_{1^n}^k$ from $\pi$
can be bounded by
\begin{eqnarray*}
&&\bigl\llVert P_x^k - \pi\bigr\rrVert
_2^2 \le\frac{1}{\pi^2(x)} \sum
_\lambda\beta_\lambda^{2k}
g_\lambda^2\bigl(1^n\bigr) -1,
\end{eqnarray*}
where we use the sloppy (but standard) notation $\llVert P_x^k -\pi
\rrVert _2$ to mean $\llVert \frac{P_x^k}{\pi}-1\rrVert _{\mathcal
{L}^2(\pi)}$.

For $\lambda= (n)$, corresponding to the trivial representation on
$S_n$, and starting point $x = (1^n)$, the summand exactly cancels
$-1$: $\beta_{(n)} = 1$, $j_{(n)} = \Pi\theta^n n!$ and $c_{(n),1^n}
=1$ (by Lemma~\ref{1^n}), whereas $\pi(1^n) = \Pi^{-1}$, so
\begin{eqnarray*}
&&\frac{1}{\pi(1^n)^2} \beta_{(n)}^{2k} g_{(n)}^2
\bigl(1^n\bigr) = 1.
\end{eqnarray*}

Thus using the explicit formula for $g_\lambda$, we immediately have
%
\begin{equation}
\label{L^2bound} \bigl\llVert P_x^k - \pi\bigr\rrVert
_2^2 = \theta^n n! \Pi_n \sum
_{\lambda
\vdash n, \lambda\neq(n)} \frac{\beta_\lambda^{2k}}{j_\lambda}.
\end{equation}
We now break the sum according to the sign of $\beta_\lambda$,
%
\begin{equation}
\label{breaksum} \sum_{\lambda\vdash n, \lambda\neq(n)} \frac{\beta_\lambda
^{2k}}{j_\lambda} = \sum
_{\beta_\lambda\ge0} ^* \frac
{\beta_\lambda^{2k}}{j_\lambda} + \sum
_{\beta_\lambda<0} \frac
{\beta_\lambda^{2k}}{j_\lambda},
\end{equation}
where $\sum^*$ denotes summation skipping the top eigenvalue indexed
by $\lambda= (n)$. Next we can rewrite the first summand on the right
according to the size of $\lambda_1$, and obtain the following bound:
\begin{eqnarray} \label{positivebeta}
\theta^n n! \Pi_n \sum_{\beta_\lambda\ge0}
^* \frac{\beta_\lambda^{2k}}{j_\lambda}\nonumber
&=& \sum_{s=n-1}^1
\sum_{\lambda\dvtx
\lambda_1 =s, \beta_\lambda\ge0} \theta^n n! \Pi_n
\frac{\beta
_\lambda^{2k}}{j_\lambda}
\nonumber\\[-8pt]\\[-8pt]
&\le& \sum_{s=n-1}^1 \max\bigl\{
\beta_\lambda^{2k}\dvtx  \beta_\lambda\ge0,
\lambda_1 =s\bigr\} \sum_{\lambda\dvtx  \lambda_1 = s}
\frac{\theta^n n!
\Pi_n}{j_\lambda}.\nonumber
\end{eqnarray}

Splitting $\Pi= \Pi_{n,\theta}$ into two subproducts, and using
Lemma~\ref{jlambdasplit}, we have
\begin{eqnarray*}
\frac{\Pi n! \theta^n}{j_\lambda} &\le&\frac{({\Pi_n}/{(\Pi_{n-\lambda_1})})
({n!}/({(n-\lambda_1)!})) \theta^{\lambda
_1}}{\theta^{2\lambda_1 -1 } \lambda_1!^2 \lambda_1^{\theta
^{-1}-1} e^{-\pi^2/12 \theta^2}} \frac{\Pi_{n-\lambda_1}
(n-\lambda_1)! \theta^{n- \lambda_1}}{j_{(\lambda_2,\ldots,\lambda_n)}},
\end{eqnarray*}
where\vspace*{1pt} the second quotient factor happens to be $\frac{1}{\pi
(1^{n-\lambda_1})^2} g_{(\lambda_2, \ldots, \lambda
_n)}^2(1^{n-\lambda_1})$; see Theorem~\ref{diagonalization} and
Lemma~\ref{1^n}. Also denote the first factor by $q_{n, \lambda_1}$.

By (\ref{leftsum}), the definition of $\Pi_{n-s}:= \pi
(1^{n-s})^{-1}$ (see Definition~\ref{definition1}), and the fact $\Pi
_n = \theta^n (n-1)! e^{\theta^{-1}\sum_{i=2}^n {1}/({i-1})}$, we
can bound the summand of the right-hand side of (\ref{positivebeta}) for a fixed
$s$ as
\begin{eqnarray*}
\sum_{\lambda\dvtx  \lambda_1 = s} \frac{\Pi_n n! \theta^n}{j_\lambda
} &\le&
q_{n,s} \sum_{\mu\vdash n-s} \frac{1}{\pi(1^{n-s})^2}
g_\mu ^2\bigl(1^{n-s}\bigr)
=  q_{n,s}
\frac{1}{\pi(1^{n-s})}
\\
& \le&\biggl(\frac
{n}{s}\biggr)^{\theta^{-1} -1}
\theta^{n-s+1} \biggl(\frac{n!}{s!}\biggr)^2
\frac
{e^{\pi^2/12\theta^2}}{(n-s)!}.
\end{eqnarray*}

We will now reduce the $\mathcal{L}^2$ bound (\ref{L^2bound}) to
bounding the following quantity:
%
\begin{equation}
\label{corebound} b_{n,+}:=\sum_{s=n-1}^1
\beta_{s,+}^{2k} \biggl(\frac{n}{s}
\biggr)^{\theta^{-1}
-1} \theta^{n-s+1} \biggl(\frac{n!}{s!}
\biggr)^2 \frac{e^{\pi^2/12\theta^2}}{(n-s)!},
\end{equation}
where $\beta_{s,+}:= \max\{\beta_\lambda\dvtx  \beta_\lambda\ge0,
\lambda_1 =s\}$.

For the second summand of (\ref{breaksum}), we obtain
%
\begin{eqnarray}
\label{secondsum} \sum_{\beta_\lambda<0} \frac{\beta_\lambda^{2k}}{j_\lambda} &\le&
\sum_{s=n-1}^1 \max\bigl\{
\beta_\lambda^{2k}\dvtx  \beta_\lambda< 0,
\lambda^t_1 =s\bigr\} \sum_{\lambda\dvtx  \lambda^t_1 = s}
\frac{\theta^n
n! \Pi_n}{j_\lambda}
\nonumber\\[-8pt]\\[-8pt]
&&{} + \beta_{1^n}^{2k} \frac{\theta^n n! \Pi_n}{j_{1^n}}.\nonumber
\end{eqnarray}
Using the explicit formula (\ref{betalambda}) for $\beta_\lambda$,
we get
\begin{eqnarray*}
\beta_{1^n} &=& 1 - (\theta\wedge1) - \bigl(\theta^{-1} \wedge1
\bigr) \in(-1,1),
\end{eqnarray*}
for $\theta\neq1$. On the other hand,
\begin{eqnarray*}
\Pi_n n! \theta^n / j_{1^n} &=& \prod
_{i=1}^{n-1} (1 + i\theta) n! \theta^n \Big/
\prod_{i=1}^n (i \theta) \bigl( 1 + (i-1)
\theta\bigr)=1,
\end{eqnarray*}
since $j_{1^n} = \prod_{i=1}^n (i \theta) ( 1 + (i-1) \theta)$ by definition.
Thus
\begin{eqnarray*}
\beta_{1^n}^{\Omega(n)} \Pi_n n! \theta^n /
j_{1^n} &=& o(1),
\end{eqnarray*}
which has negligible contribution in (\ref{L^2bound}).

For the remaining terms in (\ref{secondsum}), first observe that
\begin{eqnarray*}
j_{\lambda^t} &\ge&\prod_{i=1}^{\lambda_1} i
\prod_{i=1}^{\lambda
_1 -1} (i + \theta)
j_{(\lambda_2, \ldots, \lambda_n)^t}.
\end{eqnarray*}
Hence when $\theta< 1$,
\begin{eqnarray*}
j_{\lambda^t} &\ge&\prod_{i=1}^{\lambda_1} (i
\theta) \prod_{i=1}^{\lambda_1 -1} (i\theta+ 1)
j_{(\lambda_2, \ldots, \lambda_n)^t}
\end{eqnarray*}
and using the bound $|\beta_\lambda| \le\beta_{\lambda^t}$, for
$\beta_\lambda< 0$, we get
\begin{eqnarray*}
&&\sum_{\beta_\lambda<0} \frac{\beta_\lambda^{2k}}{j_\lambda} \le \sum^*
_{\beta_\lambda\ge0} \frac{\beta_\lambda
^{2k}}{j_\lambda} + o(1).
\end{eqnarray*}
If $\theta> 1$, the $j_{\lambda^t}$ is comparable to $j_\lambda$
within an exponential factor
\[
j_{\lambda^t} \ge j_\lambda\theta^{2n}.
\]
Furthermore by the explicit formula of $\beta_\lambda$, we have
\begin{eqnarray*}
-\beta_\lambda(\theta) &=& -\biggl(\frac{n(\lambda^t)}{{n\choose 2}} -
\theta^{-1} \frac{n(\lambda)}{{n\choose 2}}\biggr) \le\theta^{-1}
\frac{n(\lambda)}{{n\choose 2}} - \frac{n(\lambda^t)}{
{n\choose 2}}
\\
& \le& \theta^{-1}\biggl(
\frac{n(\lambda) - n(\lambda^t)}{
{n\choose 2}}\biggr)
\le - \theta^{-1}\beta_\lambda(1).
\end{eqnarray*}
So since $k:= \frac{1}{2(\theta\wedge1)} n (c + \log n) = \Theta(n
\log n)$, $\theta^{-2k} \theta^{2n} = o(1)$. Thus we can still
compare the negative $\beta_\lambda$ sum to the positive one,
\begin{eqnarray*}
&&\sum_{\beta_\lambda<0} \frac{\beta_\lambda^{2k}}{j_\lambda} \le
b_{n,+}+ o(1).
\end{eqnarray*}

It remains to bound $2 b_{n,+}$. First note that the factor $2$ in
front is immaterial, since for $\lambda\neq(n)$,
\[
\beta_\lambda^{c n} < \beta_{(n-1,1)}^{cn} \le
e^{-\Omega_\theta(c)},
\]
thanks\vspace*{1pt} to the monotonicity of $\beta_\lambda$'s. So by increasing $c$
in $k= \frac{1}{2 (\theta\wedge1)} n (c + \log n)$, we can decrease
$b_{n,+}$ by a factor of $2$. The factor $c_\theta= e^{\pi^2/2 \theta
^2}$ can be ignored similarly. We can also get rid of the factor
$(\frac{n}{s})^{\theta^{-1}-1}$ in (\ref{corebound}) as follows.

For $s \ge n/2$, $(\frac{n}{s})^{\theta^{-1}-1} = \mathcal{O}_\theta
(1)$, so again increasing $c$ annihilates it. For $s < n/2$, recall the
second bound on $\beta_\lambda$ (\ref{generallambda}), which
implies that for $\lambda_1 < n/2$, $\beta_\lambda$ is bounded away
from $1$ uniformly in $n$. Now in the definition of $b_{n,+}$, $\beta
_\lambda$ is assumed to be nonnegative (alternatively, $\beta_{1^n}$
is bounded uniformly away from $-1$), hence raising $\beta_\lambda$
to the power $\Omega(n)$ easily cancels any power of $n$, that is,
\[
n^{\theta^{-1} -1} \beta_\lambda^{cn} = o(1).
\]
So together, by increasing $c$, we can reduce the problem to bounding
the following quantity:
\[
B_{n,+}:=\sum_{s=n-1}^1
\frac{\theta^{n-s+1}}{(n-s)!} \biggl(\frac
{n!}{s!}\biggr)^2
\beta_{s,+}^{2k}.
\]

The only estimates we rely on now are (\ref{biglambda1}) and (\ref{generallambda}) from Lemma~\ref{technical}; the idea will be similar
to \cite{DiaconisShahshahani}; see also \cite{PD}. First note that it
suffices to show
%
\begin{equation}
\label{individualterm} \frac{\theta^{n-s+1}}{(n-s)!} \biggl(\frac{n!}{s!}\biggr)^2
\beta_{s,+}^{2k} = \mathcal{O}(1),
\end{equation}
uniformly for all $s \in[1,n-1]$ and $c$ sufficiently large. Indeed
using (\ref{generallambda}), we have
\[
\beta_{s,+} \le e^{-x - x^2/2} \le e^{-x}
\]
for $x = (\theta\wedge1) \frac{n-s}{n}$. Therefore
\[
\mathcal{O}(1) \sum_{s=n-1}^1
\beta_{s,+}^{(cn)/(2(1\wedge
\theta))} \le\mathcal{O}(1) \sum
_{t=1}^{n} e^{-tc} = o_c(1),
\]
by geometric summation; in fact, using (\ref{biglambda1}) we can get
a better bound, but that's not necessary.

Next recall (\ref{biglambda1}) as well as the estimates (no Stirling
formula needed)
\begin{eqnarray*}
\frac{n!}{s!} &\le& e^{\int_s^n \log x \,dx + \log n - \log s} = e^{n
\log n - s \log s -(n-s) + \log n - \log s},
\\
(n-s)! &\ge& e^{\int_1^{n-s} \log x \,dx} = e^{(n-s) \log(n-s) - (n-s -1)}.
\end{eqnarray*}
Taking logarithm, and letting $s = \alpha n$, we can bound the
left-hand side of (\ref{individualterm}) by
%
\begin{eqnarray}\label{alphaexponent}
&& \log\biggl[ \frac{\theta^{n-s+1}}{(n-s)!} \biggl(\frac{n!}{s!}\biggr)^2
\beta _{s,+}^{2k}\biggr]\nonumber
\\
&&\qquad \le(1-2 \alpha) (1-\alpha) n \log n
- (1-\alpha)n \log (1-\alpha)
\\
&&\quad\qquad{}- 2 \alpha n \log\alpha- 2 \log\alpha+ \bigl(C_1(\theta) -2 c \alpha\bigr) (1-\alpha) n +
C_2(\theta),\nonumber
\end{eqnarray}
where $C_1(\theta), C_2(\theta)$ are constants that depend only on
$\theta$.

For $\alpha\ge\alpha_0 \in(1/2,1)$, the right-hand side of (\ref
{alphaexponent}) can be further simplified to
\[
(1-\alpha) n\bigl[ (1- 2\alpha) \log n - \log(1-\alpha) +C_1'(
\theta) -c \bigr] + C_2'(\theta).
\]
Since $(1-\alpha)n \ge1$, and we can choose $c$ as large as we want,
it suffices to show the expression inside the square brackets above is
$\mathcal{O}(1)$. But when $\alpha= 1-1/n$  or~$1/2$, this is clearly
true. Furthermore, the derivative
\[
\frac{d}{d\alpha} \bigl[(1-2 \alpha) \log n - \log(1-\alpha)\bigr] = -2 \log n
+ \frac{1}{1-\alpha}
\]
is monotone increasing, showing that the function $\alpha\mapsto(1- 2
\alpha) \log n - \log(1-\alpha)$ is convex, and its value for any
$\alpha\in[1/2, 1-1/n]$ is bounded above by the values at the
boundary points.

Next let $\alpha< \alpha_0$. Using the second bound for $\beta
_\lambda$, (\ref{generallambda}), we have
\[
\beta_\lambda\le e^{-(\theta\wedge1) (1- ({s}/{n}))}.
\]
%
Then the logarithm of the left-hand side of (\ref{individualterm}) is
bounded by
%
\begin{eqnarray}\label{smallalphaexponent}
&& (1- \alpha) n \log n + (1-\alpha) \bigl(\log\theta- \log(1-\alpha) +1\bigr) n
- 2 \alpha n \log\alpha
\nonumber
\\
&&\quad{} + 2 \log\alpha- (1-\alpha) n \log n - (1-\alpha)c n
\\
&&\qquad \le(1-\alpha) \bigl(C(
\theta) - c\bigr) n + 2 (1- \alpha n) \log\alpha.\nonumber
\end{eqnarray}
Clearly $(1 -\alpha n)\log\alpha= \mathcal{O}(n)$ for $\alpha\in
[\frac{1}{n}, \alpha_0]$. So for sufficiently large $c$, the
right-hand side above goes to $-\infty$. Together this shows (\ref
{individualterm}) is true for all $s \in[\frac{1}{n}, 1 - \frac
{1}{n}]$, and concludes the upper bound for the mixing time.
%
%
%

\section{Mixing time lower bound}
We rely heavily on results from \cite{StanleyJack}.
Again we collect some notation needed in the analysis below:

\begin{defn}
For $\lambda, \rho\vdash n$, let:
\begin{itemize}
\item$H_\lambda$ be the product of all hook-lengths of the Ferrers
diagram for $\lambda$;
\item$z_\rho:= \prod_{i=1}^n i^{m_i} m_i!$ and $m_i =m_i(\rho)$ is
the number of parts in $\rho$ of length $i$;
\item$\chi_\lambda(\rho)$ be the character of $S_n$ indexed by
$\lambda$ evaluated at an element of conjugacy class $\rho$;
alternatively, they can be defined by the system
\[
p_\rho= \sum_\lambda\chi_\lambda(
\rho) s_\lambda,
\]
where $s_\rho$ are the Schur polynomials.
\end{itemize}
\end{defn}

\emph{Warning}. Note the Schur polynomials are not direct specializations of
the Jack polynomials; they differ by a factor
%
\begin{equation}
\label{SchurJack} s_\lambda= H_\lambda^{-1}
J_\lambda(1).
\end{equation}
Thus the matrix $\chi_\lambda(\rho)$ is the inverse of $H_\lambda
^{-1} c_{\lambda,\rho}(1)$.

%
\begin{lem}[(\cite{Macdonald}, Chapter I, equation (7.5); see also \cite{StanleyJack}, equation (50))]\label{cchirelation}
The re\-lation between $c_{\lambda,\rho}(1)$ and $\chi_\lambda(\rho
)$ is given by
\begin{eqnarray*}
c_{\lambda,\rho}(1) &=& H_\lambda z_\rho^{-1}
\chi_\lambda(\rho).
\end{eqnarray*}
\end{lem}

%
\begin{cor} \label{Schurpower}
The inverse matrix to $\chi_\lambda(\rho)$ is given by $\chi
_\lambda(\rho) z_\rho^{-1}$, that is,
\begin{eqnarray*}
s_\lambda&=& \sum_\rho\chi_\lambda(
\rho) z_\rho^{-1} p_\rho.
\end{eqnarray*}
\end{cor}

\begin{pf}
By relation (\ref{SchurJack}) and the lemma above, we have
\begin{eqnarray*}
s_\lambda&=& H_\lambda^{-1} J_\lambda(1)
\\
&=& H_\lambda^{-1} \sum_\rho
H_\lambda z_\rho^{-1} \chi_\lambda (\rho)
p_\rho.
\end{eqnarray*}
Comparing the coefficients with (\ref{Jackpower}) in Definition~\ref{definition1} yields the result.
\end{pf}

As in the $\theta=1$ case studied by Diaconis and Shahshahani, the
strategy will be to use a certain eigenfunction $f$ of the chain as
test function and compare the probabilities of the event $\{f < \eta\}
$ for some $\eta\in\mathbb{R}$ under the stationary distribution and
the distribution at time slightly before the mixing time, which in our
case is $k(c):= \frac{1}{2 (\theta\wedge1)}n (\log n -c)$.

In the case where $\theta=1$, $\rho\mapsto\chi_{(n-1,1)}(\rho) =
m_1(\rho) -1$ is the desired eigenfunction. So it is natural to guess
that a suitable affine transformation of the fixed-point (aka
$1$-cycle) counting function $\rho\mapsto m_1(\rho)$ is the desired
eigenfunction.

Lemma~\ref{cchirelation} shows that the following normalized version
of $c_{\lambda, \rho}$ is the right analogue of characters of the
symmetric group
%
\begin{equation}
\label{dlambda,rho} d_{\lambda, \rho}(\theta):= c_{\lambda, \rho}(\theta)
z_\rho \theta^{-(n-\ell(\rho))} H_{\lambda}^{-1}.
\end{equation}

Thus our candidate test function will be $d_\lambda(\rho) =
d_{\lambda,\rho}$.\vadjust{\goodbreak}

It is straightforward to compute $\mathbb{P}_\infty(d_\lambda< \eta
)$ where $\mathbb{P}_\infty$ denotes the stationary measure; the
number of cycles are asymptotically independent and Poisson
distributed. To estimate $\mathbb{P}_k(d_\lambda< \eta)$, one uses
the second moment method. The first moment of $d_{(n-1,1)}$ is easily
computed since it is proportional to the right eigenfunctions of the
chain $P_\theta$; see Corollary~\ref{righteigenfunctions}. For
second moments, we need to decompose $d_{(n-1,1)}^2$ as\vspace*{1pt} linear
combinations of other $d_\lambda$'s. This is accomplished by first
expressing $d_{(n-1,1)}$ and other $d_\lambda$'s in terms of powers of
$m_i$'s, the number of $i$-cycles (not to be confused with monomial
symmetric functions), then deducing the relationship by solving the
appropriate system of linear equations. The analysis below will be an
elaboration of this strategy.

First we need:

\begin{prop}[(\cite{StanleyJack}, Proposition~7.5)]
Let $d_{\lambda,\rho} = d_{\lambda,\rho}(\theta)$ be defined as
above. Then
\begin{eqnarray*}
&& \theta^{k+1} (k+n) d_{(n, 1^k),\rho} (\theta)
\\
&&\qquad = \sum
_{j=0}^k (-1)^{k-j} \bigl(j + (n+k-j)
\theta\bigr) \sum_{\nu\vdash j} \biggl[\prod
_i \pmatrix{m_i(\rho)
\cr
m_i(\nu)}
\biggr] (-1)^{j-\ell(\nu)} \theta^{\ell
(\nu)}.
\end{eqnarray*}
\end{prop}

Note that the partitions in the proposition above is for $n+k$, rather
than $n$.

From this, we easily obtain
%
\begin{equation}
\label{d{n-1,1,rho}} d_{(n-1,1),\rho}(\theta) = -\frac{1}{\theta} +
\frac{1 +
(n-1)\theta}{\theta n} m_1(\rho)
\end{equation}
and
%
\begin{eqnarray} \label{d{n-2,1^2,rho}}
d_{(n-2,1^2),\rho} (\theta) &=& \frac{1}{\theta^2} - \biggl(\frac
{1+(n-1)\theta}{n \theta^2} +
\frac{2 + (n-2)\theta}{2\theta n}\biggr) m_1(\rho)
\nonumber\\[-8pt]\\[-8pt]
&&{} + \frac{2+(n-2)\theta}{2 \theta n} m_1(\rho)^2 - \frac{2 +
(n-2)\theta}{\theta^2 n}
m_2(\rho).\nonumber
\end{eqnarray}

Note in particular,
%
\begin{eqnarray}
\chi_{(n-1,1)}(\rho) &=& m_1(\rho) -1,
\\
\chi_{(n-2,1^2)}(\rho) &=& 1 - \tfrac{3}{2} m_1(\rho) +
\tfrac{1}{2} m_1(\rho)^2 - m_2(\rho),
\end{eqnarray}
as expected.

Using the Schur--Weyl relation
\begin{eqnarray*}
\chi_{n-1,1}^2 &=& \chi_n + \chi_{n-1,1} +
\chi_{n-2,2} + \chi_{n-2,1^2},
\end{eqnarray*}
and we also obtain
\[
\chi_{(n-2,2)} = \tfrac{1}{2} m_1^2 -
\tfrac{3}{2} m_1 + m_2.
\]

To get $J_{(n-2,2)}$, we need the conjecture right after Proposition~7.2 as well as Corollary~3.5 from \cite{StanleyJack}. The conjecture
has been proved in \cite{Koike}. Notice the parameter $\alpha$ is the
same as our parameter $\theta$.

%
\begin{prop}[(\cite{StanleyJack}, Proposition~7.2)]
The Jack polynomials corresponding to the partition $(2^i, 1^j)$ have
the following expansion in terms of the monomial symmetric basis
$m_\lambda$:
\begin{eqnarray*}
J_{(2^i,1^j)} &=& \sum_{r=0}^i
(i)_r (\theta+ i + j)_r \bigl(2(i-r) + j\bigr)!
m_{(2^r, 1^{2(i-r) + j})},
\end{eqnarray*}
where $(i)_r:= i(i-1) \cdots(i-r+1)$.
\end{prop}

%
\begin{prop}[(\cite{Koike}, Theorem 1.1)]
In terms of Schur polynomial basis, we have
\begin{eqnarray*}
J_{(2^i, 1^j)} &=& \sum_{r=0}^i
(i)_r (\theta+ i + j)_r (i - r - \theta )
_{i-r} (i+j-r)! s_{(2^r, 1^{2(i-r)+j})}.
\end{eqnarray*}
\end{prop}

The next result relates Jack polynomials corresponding to conjugate
partitions, when expressed in terms of the power sum basis.

\begin{prop}[(\cite{StanleyJack}, Corollary~3.5)]
Let $J_\lambda= \sum_\mu c_{\lambda\mu}(\theta) p_{\mu}$, then
\begin{eqnarray*}
J_{\lambda^t} &=& \sum_\mu(-
\theta)^{n-\ell(\mu)} c_{\lambda\mu
}\biggl(\frac{1}{\theta}\biggr)
p_\mu.
\end{eqnarray*}
\end{prop}

We also recall from Corollary~\ref{Schurpower} that $s_\lambda= \sum_\rho\chi_\lambda(\rho) z_\rho^{-1} p_\rho$, where $\chi_\lambda
(\rho)$ is the character of $\lambda$ evaluated at $\rho$.

Combining the previous results, we easily get
\begin{eqnarray*}
\hspace*{-4pt}&& J_{n-2,2} (\theta)
\\[-2pt]
\hspace*{-4pt}&&\qquad = \sum_\rho(-
\theta)^{n-\ell(\rho)} \biggl[\biggl(2-\frac
{1}{\theta}\biggr) \biggl(1-
\frac{1}{\theta}\biggr) (n-2)\chi_{\rho}^{1^n}
\\[-2pt]
\hspace*{-4pt}&&\hspace*{101pt}{} + 2\biggl(n-2 +
\frac{1}{\theta}\biggr)
\biggl(1-\frac{1}{\theta}\biggr) (n-3)!\chi_{2^1,1^{n-2}}({\rho})
\\[-2pt]
\hspace*{-4pt}&&\hspace*{101pt}{} + 2\biggl(n-2 + \frac{1}{\theta}\biggr) \biggl(n-3 + \frac{1}{\theta}\biggr) (n-4)!
\chi_{2^2,
1^{n-4}}({\rho})\biggr] z_\rho^{-1}
p_\rho,
\end{eqnarray*}
where $\chi_\lambda(\rho)$ is the irreducible character $\lambda$
evaluated at $\rho$.
Therefore we can read off the coefficients
%
\begin{eqnarray}\label{c{n-2,2,rho}}
&& c_{(n-2,2),\rho}(\theta)\nonumber
\\[-2pt]
&&\qquad = \frac{(-\theta)^{n-\ell(\rho)}}{z_\rho
} \biggl(\biggl(2-
\frac{1}{\theta}\biggr) \biggl(1-\frac{1}{\theta}\biggr) (n-2)!
\chi_{1^n}(\rho )
\nonumber\\[-9pt]\\[-9pt]
&&\hspace*{91pt}{} + 2\biggl(n-2+\frac{1}{\theta}\biggr) \biggl(1-\frac{1}{\theta}\biggr)
(n-3)! \chi_{2^1,1^{n-2}}(\rho) \nonumber
\\[-2pt]
&&\hspace*{91pt}{}+ 2\biggl(n-2 + \frac{1}{\theta}\biggr)
\biggl(n-3 + \frac{1}{\theta
}\biggr) (n-4)! \chi_{2^2,1^{n-4}}(\rho)\biggr),\nonumber
\end{eqnarray}
and using the relation $\chi_{\lambda^t}(\rho) = \chi_\lambda(\rho
) \operatorname{sgn}(\rho)$, as well as the formula for $\chi_n$,
$\chi_{n-1,1}$, $\chi_{n-2,2}$ and $\chi_{n-2,1^2}$ derived above, we
also get
%
\begin{eqnarray}\label{d{n-2,2,rho}}
d_{(n-2,2),\rho}(\theta) &=& m_1 \frac{n-2 + ({1}/{\theta})}{(n-1)(n-2)} \biggl[(n-3)
\biggl(1-\frac{1}{\theta}\biggr) - \frac{3}{2} \biggl(n-3 +
\frac{1}{\theta}\biggr)\biggr]
\nonumber
\\
&&{}+ m_1^2 \frac{(n-2 + ({1}/{\theta}))(n-3 + ({1}/{\theta}))}{2(n-1)(n-2)}
\\
&&{} + m_2
\frac{(n-2 + ({1}/{\theta}))(n-3 + ({1}/{\theta}))}{(n-1)(n-2)}. \nonumber
\end{eqnarray}

Using $i=0$ and $j=n$, one gets
\begin{eqnarray*}
J_{(n)}(\theta) &=& \sum_{\rho}
\theta^{n-\ell(\rho)} n! z_\rho ^{-1} p_\rho.
\end{eqnarray*}
Hence
%
\begin{equation}
\label{c{n,rho}} c_{(n),\rho}(\theta) = \theta^{n-\ell(\rho)} n!
z_{\rho}^{-1},
\end{equation}
and
%
\begin{equation}
\label{d{n,rho}} d_{(n),\rho}(\theta) = 1.
\end{equation}

To mimic the case of $\theta=1$, we need to express $d_{(n-1,1),\rho
}(\theta)^2$ in terms of the other $d_\lambda$'s. First using (\ref
{d{n-1,1,rho}}), we get
\begin{eqnarray*}
d_{(n-1,1),\rho}(\theta)^2 &=& \frac{1}{\theta^2} + \biggl(
\frac{1}{\theta
n} + \frac{n-1}{n}\biggr)^2
m_1^2 - \frac{2}{\theta} \biggl(\frac{1}{\theta n} +
\frac{n-1}{n}\biggr) m_1.
\end{eqnarray*}
Using $m_1, m_1^2, m_2$ and $1$ as a basis, we can write
%
\begin{equation}
\label{secondmoment} \qquad\quad d_{(n-1,1),\rho}(\theta)^2 = u + v
d_{(n-1,1),\rho}(\theta) + w d_{(n-2,1^2),\rho}(\theta) + x d_{(n-2,2),\rho}(
\theta)^2,
\end{equation}
for some indeterminates $u,v,w,x$.

Comparing coefficients of the $m_i^k$'s, we get the following four equations:
%
\begin{eqnarray}
&& u - \frac{v}{\theta} + \frac{w}{\theta^2} = \frac{1}{\theta^2},
\\
&& v\biggl(\frac{1}{\theta n} + \frac{n-1}{n}\biggr) - w \biggl(
\frac{1 + (n-1) \theta
}{n \theta^2} + \frac{2 + (n-2)\theta}{2 \theta n}\biggr)\nonumber
\\
&&\quad{} + x\frac{n-2 + ({1}/{\theta})}{(n-1)(n-2)} \biggl[(n-3) \biggl(1-\frac
{1}{\theta}\biggr) - \frac{3}{2} \biggl(n-3 + \frac{1}{\theta}\biggr)\biggr]
\\
&&\qquad
= -\frac{2}{\theta} \biggl(\frac{1}{\theta n} + \frac{n-1}{n}\biggr),\nonumber
\\
&& w\frac{2 + (n-2)\theta}{2\theta n} + x \frac{(n-2 + ({1}/{\theta}))( n -3 + ({1}/{\theta}))}{2(n-1)(n-2)} = \biggl(\frac{1}{\theta n} +
\frac{n-1}{n}\biggr)^2,\hspace*{-35pt}
\\
&& - w \frac{2 + (n-2) \theta}{\theta^2 n} + x \frac{(n-2 + ({1}/{\theta}))( n-3 + ({1}/{\theta}) )}{(n-1)(n-2)} =0.
\end{eqnarray}

Solving, we get
\begin{eqnarray*}
u &=& \frac{(n^2-4 n+3) \theta^3+(n-1) \theta^2+(n-1)^2 \theta
+n-3}{\theta^2
(\theta+1) (\theta(n-3)+1) n}
\\
& =& \frac{1}{\theta^2 (\theta
+1)}+\frac{1}{\theta+1} + \mathcal{O}
\biggl(\frac{1}{n}\biggr),
\\
v &=& \frac{(n^3-6 n^2+11 n-6) \theta^3+2 (2 n^2-7
n+4) \theta^2+(3 n+2) \theta-4}{\theta n ((n^2-5 n+6)
\theta^2+(3 n-8) \theta+2)}
\\
& =& 1 + \mathcal{O}\biggl(\frac{1}{n}\biggr),
\\
w &=& \frac{2(1 + \theta(n-1))^2}{(1+\theta)(2 + \theta(n-2)) n} = \frac{2\theta}{1+\theta} + \mathcal{O}\biggl(
\frac{1}{n}\biggr),
\\
x &=& \frac{2(1+ \theta(n-1))^2 (n-1)(n-2)}{(1+\theta) n^2(1+\theta
(2n -5) + \theta^2 (n-2)(n-3))} = \frac{2}{1+\theta} + \mathcal {O}\biggl(
\frac{1}{n}\biggr).
\end{eqnarray*}
Notice $x + w = 2 + O(\frac{1}{n})$.\vspace*{1pt}

Also by (\ref{d{n-1,1,rho}}), (\ref{d{n-2,1^2,rho}}), (\ref
{d{n-2,2,rho}}) and (\ref{d{n,rho}}), we get
\begin{eqnarray*}
d_{n,1^n} &=& 1,
\\
d_{(n-1,1),1^n} &=& n-1,
\\
d_{(n-2,1^2),1^n} &=& \frac{(n-1)(n-2)}{2},
\\
d_{(n-2,2),1^n} &=& \frac{n(n-3)}{2},
\end{eqnarray*}
which are independent of $\theta$ because of the normalization chosen
for Jack polynomials.

Finally we recall $c_\lambda(\theta)$ are eigenfunctions of $P_\theta
$, hence so are $d_\lambda(\theta)$, with eigenvalue $\beta_\lambda$.
We list the relevant eigenvalues here:
\begin{eqnarray*}
\beta_{(n)} &=& 1,
\\
\beta_{(n-1,1)} &=& 1- (\theta\wedge1) + \frac{\theta
{n-1\choose 2} -1}{(\theta\vee1) {n\choose 2}} = 1-(1 \wedge
\theta ) \frac{2}{n} + \mathcal{O}\biggl(\frac{1}{n^2}\biggr),
\\
\beta_{(n-2,1^2)} &=& 1 - (\theta\wedge1) + \frac{\theta
{n-2\choose 2} -3}{(\theta\vee1) {n\choose 2}} = 1-(\theta
\wedge1) \frac{4}{n} + \mathcal{O}\biggl(\frac{1}{n^2}\biggr),
\\
\beta_{(n-2,2)} &=& 1-(\theta\wedge1) + \frac{\theta
{n-2\choose 2} -2}{(\theta\vee1) {n\choose 2}} = 1 - (\theta
\wedge 1) \frac{4}{n} + \mathcal{O}\biggl(\frac{1}{n^2}\biggr).
\end{eqnarray*}
Notice
\[
\lim_{n\to\infty} \frac{(1-\beta_{(n-1,1)})}{1-\beta_{(n-2,1^2)}} = \lim_{n\to\infty}
\frac{(1-\beta_{(n-1,1)})}{1-\beta_{(n-2,2)}} = 2.
\]

Now\vspace*{2pt} we are fully equipped to prove the lower bound.
First observe\break $\mathcal{L}_\infty(d_{(n-1,1)}(\theta)) \prec
\operatorname{Poi}(\theta^{-1}) + 1$, which comes from Feller
coupling. Here $d_{(n-1,1)}$ stands for the random variable
$d_{(n-1,1),\rho}$ where $\rho$ has Ewens sampling distribution with
parameter $\theta^{-1}$, as indicated by the subscript $\infty$. Therefore,
\[
\lim_{\eta\to\infty}\mathbb{P}_\infty\bigl(d_{(n-1,1)}(
\theta) \le \eta\bigr) =1.
\]
Furthermore,
\[
\mathbb{P}_k \bigl(d_{(n-1,1)}(\theta) \le\eta\bigr) \le
\frac{\operatorname{var}_k(d_{(n-1,1)})}{(\eta- \mathbb{E}_k(d_{(n-1,1)}))^2}.
\]
Let $k = \frac{1}{2} (\theta^{-1} \vee1) n (\log n -c)$ for any $c > 0$.

Since $d_\lambda$ are eigenfunctions, we can compute the mean and
variance at time $k$,
\begin{eqnarray*}
\mathbb{E}_k d_{(n-1,1)} &=& (n-1) \biggl(1-(\theta\wedge1)
\frac{2}{n}\biggr)^k + O(1) = e^c +
\mathcal{O}(1),
\\
\operatorname{var}_k d_{(n-1,1)} &=& \mathbb{E}_k
d_{(n-1,1)}^2 - (\mathbb{E}_k d_{(n-1,1)})^2
\\
&=& u + (n-1)v e^{-(\theta\wedge1) ((2k)/n)(1+ \mathcal{O}({1}/{n}))}
\\
&&{}+ \biggl(\frac{(n-1)(n-2)}{2}w + \frac{n(n-3)}{2}x\biggr) e^{-(\theta\wedge1)
((4k)/n)(1+ O({1}/{n}))}
\\
&&{} - (n-1)^2 e^{-(\theta\wedge1)((4k)/n)(1+ \mathcal{O}(1/n))}
\\
&\le&\frac{1}{1+\theta} + \frac{1}{\theta^2(1+\theta)} + (n-1)e^{-(\theta\wedge1) ((2k)/n)} +
\mathcal{O}\biggl(\frac
{1}{n}\biggr)
\le \mathcal{O}\bigl(e^c
\bigr).
\end{eqnarray*}

Therefore if we let $\eta= \frac{1}{2} e^c$, then
\begin{eqnarray*}
\lim_{c\to\infty}\liminf_n
\mathbb{P}_{k(c)} [d_{(n-1,1)} < \eta] &\le&\lim_{c\to\infty}
\liminf_n \mathcal{O}(1)\frac{e^c}{ ((1/2) e^c + \mathcal{O}(1))^2} = 0.
\end{eqnarray*}
Thus
\begin{eqnarray*}
&& \lim_{c \to\infty} \lim_{n \to\infty} \bigl\llVert
\delta_{1^n} P^{k(c)} - \pi\bigr\rrVert _\mathrm{TV}
\\
&&\qquad \ge \lim_{c \to\infty} \liminf_n \biggl|
\mathbb{P}_{k(c)}\biggl[d_{(n-1,1)} < \frac{1}{2}
e^c\biggr] - \mathbb {P}_\infty\biggl[d_{(n-1,1)} <
\frac{1}{2} e^c\biggr]\biggr|=1.
\end{eqnarray*}

%
\begin{remark}
Wilson's method gives a suboptimal lower bound, once we know the
``geometric'' information that $d_{(n-1,1),\rho} = -\frac{1}{\theta}
+ \frac{1 + (n-1)\theta}{\theta n} m_1(\rho)$:
let $X_1$ be the random variable distributed as $\delta_x P$. We have
\begin{eqnarray*}
R &:=& \sup_{x \in S_n} \mathbb{E}(d_{(n-1,1),X_1} -
d_{(n-1,1),x})^2 \le 1,
\\
\log\frac{1}{\beta_{(n-1,1)}} &=& (\theta\wedge1)\frac{2}{n} + O\biggl(
\frac{1}{n^2}\biggr)
\end{eqnarray*}
and $d_{(n-1,1),1^n} = n-1$. Hence by Wilson \cite{Yuval-book},
\begin{eqnarray*}
t_{\mathrm{mix}}(\varepsilon) &\ge&\frac{1}{2 \log(1/\beta_{(n-1,1)})} \biggl[ \log\biggl[
\frac{(1-\beta_{(n-1,1)}) d_{(n-1,1),1^n}^2}{2R}\biggr] + \log\frac
{1 -\varepsilon}{\varepsilon}\biggr]
\\
&\ge&\frac{n}{4(1 \wedge\theta)} \log n + \log\varepsilon^{-1} + \mathcal{O}(1).
\end{eqnarray*}
This misses a factor of $2$ from the lower bound obtained by second
moment method. The discrepancy is possibly due to the nonlocal nature
of the random transposition walk.
\end{remark}

\begin{appendix}\label{app}
\section{Sekiguchi--Debiard operator over other bases}
Having seen the probabilistic interpretation of the second order
differential operator (\ref{secondorderLaplace-Beltrami}) expressed
in the power sum symmetric basis $p_\lambda$, it is natural to
consider the following:

%
\begin{question}
Are there other bases of the symmetric polynomials $\Lambda_n$ over
which $L^2_\theta$ has natural probabilistic interpretation?
\end{question}

Here we examine the remaining four fundamental bases: monomial,
elementary and complete. The action of $L_\theta^2$ on the monomial
basis $m_\lambda$ is well known to be strictly upper triangular (\cite{Macdonald}, page 317), when the rows and columns of the Markov matrix
are arranged in a total order compatible with the natural partial order
on the set of partitions $\mathcal{P}_n$ of $n$: $\mu< \lambda$ if
$\mu_1 + \cdots+ \mu_r \le\lambda_1 + \cdots+ \lambda_r$ for all
$r$. In particular, if $L^2_\theta$ does define a Markov matrix
(meaning the entries are nonnegative), it has a single absorbing state
at $(1^n)$.\looseness=-1

Next consider the action of $L_\theta^2$ on $e_\lambda$, the
elementary symmetric polynomials. This has been studied in detail in
\cite{Beerends}. Here we give a quick development avoiding lengthy
computations. The action of $U$ is easy to describe. For any simple
elementary polynomial $e_r$, the operators $x_i \partial_i$ and $(x_i
\partial_i)^2$ simply collect all the terms in~$e_r$ that contains the
factor $x_i$. So after summing over $i \in[n]$, this results in a
constant multiple of the identity. Thus to get a nontrivial action, one
must consider a composite $e_{r_1,r_2}:=e_{r_1} e_{r_2}$. In this case,
one can show that for $r_1 \le r_2$,
\begin{eqnarray*}
U (e_{r_1,r_2}) &=& 3(1+r_1)e_{r_1,r_2} -\sum
_{j=0}^{r_1 -1} 2(r_1 + r_2
-2j)e_{r_1+r_2-j,j}.
\end{eqnarray*}
Thus $U$ is strictly lower triangular with respect to the partial order
$\preceq$. It turns out that the action of $V$ on the $e_\lambda$ is
diagonal: first of all $V$ satisfies a product rule on $e_\lambda=
\prod_{j=1}^{\ell(\lambda)} e_{\lambda_j}$; by pairing up $j \neq
k$, one also sees that $V e_r$ consists of monomials with no repeated
factors, hence by symmetry must be a multiple of $e_r$. Thus the
following linear combination yields a legitimate Markov matrix:
\[
M_e(c_1,c_2):=c_1 I +
c_2 \biggl(\frac{\theta}{2} U + V\biggr).
\]
Notice that we need to add the multiple of $I$ to make sure that the
diagonal entries of $M_e$ are nonnegative. Also observe that the Jack
parameter $\theta$ needs to be nonpositive for the off-diagonal
entries to be nonnegative. It is clear from the description of $U$ and
$V$ that this Markov chain is absorbing at $(n)$, because the next step
either stays in the current state or goes to a state corresponding to a
partition bigger than that of the current state.

Next we look at the complete symmetric polynomials $h_\lambda$. First
consider the action of $\langle X, \nabla\rangle$ on $h_r$, one of
the generators. Since $h_r = s_{(r)}$ is a degree $r$ homogeneous
polynomial, the action of $\langle X, D\rangle= \sum_{i=1}^n x_i
\partial_i$ is simply multiplication by~$r$; that is, any homogeneous
polynomials are eigenfunctions $\langle X, D\rangle$.
However, the operator $(\langle X, D \rangle)^2$ acts nontrivially on $h_r$,
\[
\sum_{i=1}^n (x_i
\partial_i)^2 h_4 = -2 h_{1^4}+10
h_{2,1,1}-8 h_{2^2}-12 h_{3,1}+28 h_4.
\]
For partitions of more than one part, the computation gets unwieldy,
and I have not tried to express $U(h_{r_1}h_{r_2})$ and
$V(h_{r_1}h_{r_2})$ in terms of $h_\lambda$ explicitly\vadjust{\goodbreak} because of the
following numerical observation: for $\lambda= (3,2,1)$, we have
\begin{eqnarray*}
U h_\lambda&=& 2h_{2,1^4} - 8h_{2^2,1^2} -
2h_{3,1^3} + 14 h_{3,2,1} + 6 h_{3^2} +
6h_{4,1^2} + 8h_{4,2} + 10h_{5,1},
\\
V h_\lambda&=& -h_{2,1^4} + 4h_{2^2,1^2} +
h_{3,1^3} + 32 h_{3,2,1}.
\end{eqnarray*}
The only linear combination of the above two expressions that yields
nonnegative coefficients is $\frac{1}{2} U + V$, which\vspace*{1pt} corresponds to
$\theta=1$. But in that case, the Markov chain is again absorbing at
(\ref{Hanlonbeta}). So we arrive at the following result:\looseness=-1

\begin{prop}
The operator $L_\theta^2$ gives a Markov matrix under the complete
symmetric polynomial basis $h_\lambda$ for all $n$ if and only if
$\theta=1$. In this case, the Markov chain never goes toward
partitions of fewer or equal parts, hence is absorbing at $(n)$.
\end{prop}

\begin{pf}
When $\theta=1$, $h_\lambda$ are dual to $e_\lambda$ with respect to
the Jacobi--Trudy identity. Hence the walk defined by $D_1^2$ on
$h_\lambda$ can\vspace*{1pt} be obtained from the upper-triangular walk on
$e_\lambda$ under the map $\lambda\mapsto\lambda^t$; in particular
the walk is absorbing at $(r)$. For $\theta\neq1$, the numerical
example above suffices to show the associated walk is not positive.
\end{pf}

For $\theta\neq1$, the resulting signed Markov matrix seems to have
nontrivial left eigenvector corresponding to the eigenvalue $1$. I
have not checked if this corresponds to some nice stationary
distribution on $\mathcal{P}_n$; presumably it will define a signed measure.

\section{Higher order Sekiguchi--Debiard operators}
Throughout this section $N$ will denote the number of underlying
variables in the symmetric functions, and $n$ will denote the weight of
partitions, consistent with previous sections.
It is possible to study higher order differential operators on $\Lambda
_N$ from the Sekiguchi--Debiard operator-valued generating function
(see \cite{Macdonald}, page 328),
%
\begin{eqnarray}\label{Sekiguchi-Debiard}
D_\theta(X) &:=& a_\delta(x)^{-1} \sum
_{w \in S_N} \varepsilon(w) x^{w
\delta}\prod_{j=1}^N \bigl(X + (w \delta)_j +
\theta x_j \partial_j\bigr)
\nonumber\\[-9pt]\\[-9pt]
&=&  \sum_{k=0}^N D_\theta^k X^{N-k}.\nonumber
\end{eqnarray}
Since the seminal work of Diaconis and Ram \cite{Aux} interpreting
$D_\theta^2$ above as the generator of an auxilliary variable Markov
chain, it has been tempting to consider the following:

\begin{question}
Does any of the higher order $D_\theta^k$'s admit natural
probabilistic interpretation?\vadjust{\goodbreak}
\end{question}

Below we give complete analysis of $D_\theta^3$, and show that the
answer is not nearly as nice as for $D_\theta^2$. To begin, it
suffices to understand the following two-parameter family of operators:
\begin{eqnarray*}
D(\lambda,\mu;h) &=& a_\delta(x)^{-1} \sum
_{w \in S_N} \varepsilon(w) \prod_{i=1}^\ell
\Biggl(\sum_{j=1}^N (w
\delta)_j^{\lambda_i} (x_j \partial_j)^{\mu_i}\Biggr)
\sum_{j=1}^N(x_j
\partial_j)^h
\\
&=&a_\delta(x)^{-1} \sum_{j_1, \ldots, j_\ell}
\prod_{i=1}^\ell (x_{j_i}
\partial_{j_i})^{\lambda_i} a_\delta(x) (x_{j_i}
\partial _{j_i})^{\mu_i} \sum_{j=1}^N
(x_j \partial_j)^h,
\end{eqnarray*}
where $\lambda,\mu$ are positive integer compositions and $\ell=
\ell(\mu) = \ell(\lambda)$.
Indeed, it is easy to see that [denoting by $(j_1, \ldots, j_k)$ all
distinct indices]
\begin{eqnarray*}
D^k_\theta&=& a_\delta(x)^{-1} \sum
_{w \in S_N} \sum_{(j_1, \ldots,j_k)} \sum
_{u=0}^k \frac{\theta^u}{u! (k-u)!} \prod
_{i=1}^{k-u} (w \delta)_{j_i} \prod
_{i=k-u+1}^k (x_{j_i}
\partial_{j_i}),
\end{eqnarray*}
which\vspace*{1pt} can be expressed as a linear combination of $D(\lambda,\mu,h)$'s with $|\lambda| + |\mu| + h= k$ and $h \neq1$; the factors of
the form $\sum_{j=1}^N (w \delta)_j^k$ or $\sum_{j=1}^N (x_j
\partial_j)^k$ all evaluate to constant by symmetry, and the operator
$\sum_{j=1}^N (x_j \partial_j)$ acts on $p_\lambda$ by constant
multiplication.

For $k=3$, we thus have three operators to consider: $D((1),(2);0)$,
$D((2),\break (1);0)$ and $D(\varnothing,\varnothing;3)$. We compute the action
of each on $p_\lambda$ below. We need the following three computations:
\begin{itemize}
\item $\displaystyle a_\delta(x)^{-1} x_i \partial_i a_\delta(x) = \sum_{j\neq i}
\frac{x_i}{x_i - x_j}$;\vspace*{5pt}

\item$\displaystyle
\begin{aligned}[t]
(x_i \partial_i)^2 p_\lambda&=& x_i \partial_i \sum_{s=1}^{\ell
(\lambda)} \frac{p_\lambda}{p_{\lambda_s}}(x_i \partial_i)
p_{\lambda_s}
\\
&=& \sum_{s=1}^{\ell(\lambda)}\frac{p_\lambda
}{p_{\lambda_s}}(x_i \partial_i)^2 p_{\lambda_s}
+ \sum_{t \neq s}
\frac{p_{\lambda}}{p_{\lambda_s}p_{\lambda_t}}\bigl[(x_i \partial_i)
p_{\lambda_s}\bigr]\bigl[(x_i \partial_i) p_{\lambda_t}\bigr]
\\
&=& \sum_{s=1}^{\ell(\lambda)} \lambda_s^2 x_i^{\lambda_s} \frac
{p_\lambda}{p_{\lambda_s}} + \sum_{s \neq t} \lambda_s \lambda_t
x_i^{\lambda_s + \lambda_t} \frac{p_\lambda}{p_{\lambda
_s}p_{\lambda_t}};
\end{aligned}$\vspace*{5pt}

\item$
\begin{aligned}[t]
\displaystyle \sum_{i=1}^N \sum_{j \neq i} \frac{x_i^{r + 1}}{x_i - x_j} &=& \frac
{1}{2}\sum_{i \neq j} \frac{x_i^{r +1} - x_j^{r +1}}{x_i - x_j} =
\frac{1}{2} \sum_{i \neq j} \sum_{u=0}^r x_i^u x_j^{r-u}
\\
&=& \displaystyle\frac
{1}{2} \mathop{\sum_{r_1 + r_2 = r\dvtx }}_{r_i \ge1} p_{r_1} p_{r_2} +
\frac{2N -r -1}{2} p_r.
\end{aligned}$
\end{itemize}

From this we have
\begin{eqnarray*}
&& D\bigl((1),(2);0\bigr) p_\lambda
\\
&&\qquad = a_\delta(x)^{-1}
\sum_{i=1}^N \bigl[x_i
\partial _i a_\delta(x)\bigr] (x_i
\partial_i)^2 p_\lambda
\\
&&\qquad = p_\lambda\Biggl(
\sum_{s=1}^{\ell(\lambda)} \frac{\lambda_s^2}{2}
\Biggl[(2N -\lambda_s -1) + \sum_{u=1}^{\lambda_s -1}
\frac{p_u p_{\lambda_s -u}}{p_{\lambda
_s}}\Biggr]
\\
&&\hspace*{51pt}{}+ \sum_{s \neq t} \frac{\lambda_s \lambda_t}{2} \Biggl[(2N-
\lambda_s -\lambda_t -1) \frac{p_{\lambda_s + \lambda_t}}{p_{\lambda_s}
p_{\lambda_t}} + \sum
_{u=1}^{\lambda_s + \lambda_t -1} \frac{p_u
p_{\lambda_s + \lambda_t -u}}{p_{\lambda_s}p_{\lambda_t}}\Biggr]
\Biggr).
\end{eqnarray*}
It is worth pointing out that the sum $ \sum_{u=1}^{\lambda_s +
\lambda_t -1} \frac{p_u p_{\lambda_s + \lambda_t -u}}{p_{\lambda
_s}p_{\lambda_t}}$ contains a constant term (when $u \in\{\lambda_s,
\lambda_t\}$).

Next we consider $D((2),(1);0)$. Again we collect some computations below:
\begin{itemize}
\item $\displaystyle x_i \partial_i p_\lambda= \sum_{s=1}^{\ell(\lambda)}
\lambda_s x_i^{\lambda_s} \frac{p_\lambda}{p_{\lambda_s}}$.
\item$\displaystyle
\begin{aligned}[t]
&& a_\delta(x)^{-1}(x_i \partial_i)^2 a_\delta(x)
\\
&&\qquad = (x_i \partial
_i)^2 \log a_\delta(x) + x_i^2 \bigl(\partial_i \log a_\delta(x)\bigr)^2
\\
&&\qquad = \sum_{j \neq i} \frac{-x_i x_j}{(x_i -x_j)^2} + \sum_{j \neq i}
\sum_{k \neq i,j} \frac{x_i^2}{(x_i - x_j)(x_i -x_k)}
+ \sum_{j \neq i} \frac{x_i^2}{(x_i -x_j)^2}
\\
&&\qquad = \sum_{j \neq i} \frac{x_i}{x_i - x_j} + \mathop{\sum_{j \neq
k\dvtx }}_{j, k \neq i} \frac{x_i^2}{(x_i - x_j) (x_i - x_k)}.
\end{aligned}
$

\item $\displaystyle
\begin{aligned}[t]
&& \sum_{i=1}^N \mathop{\sum_{j \neq k\dvtx }}_{j, k \neq i} \frac
{x_i^{r+2}}{(x_i - x_j)(x_i-x_k)}
\\[-3pt]
&&\qquad =
\mathop{\sum_{T \subset[N]\dvtx }}_{|T| = 3} \sum_{i \in T} \mathop
{\sum_{j \neq k\dvtx }}_{j, k \in T \setminus\{i\}} \frac{x_i^{r+2}}{(x_i
- x_j)(x_i - x_k)}
\\
&&\qquad = 2\mathop{\sum_{T \subset[N]\dvtx }}_{|T| = 3} s_r(T) =2 \mathop{\sum_{T \subset[N]\dvtx }}_{|T| = 3} \sum_{\lambda\vdash r} m_\lambda(T)
\end{aligned}$\vadjust{\goodbreak}

$\displaystyle\begin{aligned}[t]
&&\qquad = 2\mathop{\sum_{T \subset[N]\dvtx }}_{|T| = 3} \mathop{\sum_{\lambda
\vdash r\dvtx }}_{\ell(\lambda) \le3} m_\lambda(T)
\\
&&\qquad =2\biggl[\pmatrix{N-3\cr 0} \mathop{\sum_{\lambda\vdash r\dvtx }}_{\ell
(\lambda) = 3} m_\lambda+ \pmatrix{N-2\cr 1}
\mathop{\sum_{\lambda\vdash r\dvtx }}_{\ell(\lambda) = 2} m_\lambda+
\pmatrix{N-1\cr 2}
\mathop{\sum_{\lambda\vdash r\dvtx }}_{\ell(\lambda) = 1} m_\lambda\biggr],
\end{aligned}$\vspace*{5pt}

\noindent where $s_\lambda(T)$ denotes the Schur polynomial over the variables
indexed by $T$, and similarly for $m_\lambda$.

\item $
\begin{aligned}[t]
\mathop{\sum_{\lambda\vdash r\dvtx }}_{\ell(\lambda)=2} m_\lambda&=&
\mathop{\sum_{\lambda\vdash r\dvtx }}_{\ell(\lambda)=2} \sum_{i \neq
j} x_i^{\lambda_1} x_j^{\lambda_2} I(\lambda_1 \neq\lambda_2) +
\frac{1}{2}x_i^{\lambda_1} x_j^{\lambda_2} I(\lambda_1 = \lambda
_2)\\
&=& \frac{1}{2} \mathop{\sum_{r_1 + r_2 = r\dvtx }}_{r_i \ge1} (p_{r_1}
p_{r_2} - p_r)
= \frac{1}{2}
\mathop{\sum_{r_1 + r_2 =r\dvtx }}_{r_i \ge1} p_{r_1} p_{r_2} - \frac
{r-1}{2} p_r.
\end{aligned}
$

\item$
\begin{aligned}[t]
\mathop{\sum_{\lambda\vdash r\dvtx }}_{\ell(\lambda) = 3} m_\lambda&=&
\frac{1}{6}\biggl[
\mathop{\sum_{r_1+ r_2 + r_3 = r\dvtx }}_{r_i \ge1} p_{r_1}p_{r_2}
p_{r_3}
\\
&&\hspace*{10pt}{}- 3
\mathop{\sum_{r_1 + r_2 = r\dvtx }}_{r_i \ge1} (r_1 -1) p_{r_1} p_{r_2} +
(r-1)(r-2) p_r\biggr]
\end{aligned}
$

$
\begin{aligned}[t]
\hspace*{29pt}&=& \frac{1}{6}\mathop{\sum_{r_1+ r_2 + r_3 = r\dvtx }}_{r_i \ge1}
p_{r_1}p_{r_2} p_{r_3}
- \frac{r-2}{4}
\mathop{\sum_{r_1 + r_2 = r\dvtx }}_{r_i \ge1} p_{r_1} p_{r_2}
\\
&&{} + \frac
{(r-1)(r-2)}{6} p_r.
\end{aligned}
$
\end{itemize}
Putting everything together we have
\begin{eqnarray*}
&& D\bigl((2),(1);0\bigr) p_\lambda
\\[-1pt]
&&\qquad = p_\lambda\sum
_{s=1}^{\ell(\lambda)} \lambda_s\biggl[\biggl(N-
\frac{1+\lambda_s}{2}\biggr) \mathop{\sum_{r_1 + r_2 = \lambda_s\dvtx }}_{r_1, r_2 \ge1}
p_{r_1} p_{r_2} + \frac{1}{3} \mathop{\sum
_{r_1 + r_2 + r_3 = \lambda
_s\dvtx }}_{r_i \ge1} p_{r_1} p_{r_2}
p_{r_3}
\\[-1pt]
&&\hspace*{125pt}{}+ \biggl( (N-1) (N-\lambda_s) + \frac{(2 \lambda_s -3)(\lambda_s -1)}{6}\biggr)
p_{\lambda_s}\biggr].
\end{eqnarray*}


Finally we compute the action of $D(\varnothing,\varnothing;3)$,
\begin{eqnarray*}
D(\varnothing,\varnothing;3) p_\lambda &=& \sum_{i=1}^N
(x_i \partial _i)^2 \sum
_{s =1}^{\ell(\lambda)} \lambda_s
x_i^{\lambda_s} \frac
{p_\lambda}{p_{\lambda_s}}
\\[-1pt]
&=& \sum
_{i=1}^N (x_i \partial_i)
\Biggl[\sum_{s =1}^{\ell(\lambda)} \biggl[
\lambda_s^2 x_i^{\lambda_s}
\frac{p_\lambda}{p_{\lambda_s}} + \sum_{ t\neq s}
\lambda_s \lambda_t x_i^{\lambda_s + \lambda_t}
\frac{p_\lambda}{p_{\lambda_s}
p_{\lambda_t}}\biggr]\Biggr]
\\[-1pt]
&=&\sum_{i=1}^N \sum
_{s=1}^{\ell(\lambda)}\biggl[ \lambda_s^3
x_i^{\lambda_s} \frac{p_\lambda}{p_{\lambda_s}} + \sum
_{t \neq s}\biggl[ \lambda_s \lambda_t(
\lambda_s + \lambda_t) x_i^{\lambda_s +
\lambda_t}
\frac{p_\lambda}{p_{\lambda_s} p_{\lambda_t}}
\\[-1pt]
&&\hspace*{109pt}{} + \sum_{u \neq s,t}
\lambda_s \lambda_t \lambda_u
x_i^{\lambda_s +
\lambda_t + \lambda_u} \frac{p_\lambda}{p_{\lambda_s} p_{\lambda
_t} p_{\lambda_u}}\biggr]
\\[-1pt]
&&\hspace*{162pt}{}+ \sum_{t \neq s} \lambda_s^2
\lambda_t x_i^{\lambda_s + \lambda
_t} \frac{p_\lambda}{p_{\lambda_s} p_{\lambda_t}}\biggr]
\\[-1pt]
&=& p_\lambda \Biggl[\sum_{s =1}^{\ell(\lambda)}
\lambda_s^3 + 3 \sum_{t \neq s}
\lambda_s^2 \lambda_t \frac{p_{\lambda_s + \lambda_t}}{p_{\lambda
_s} p_{\lambda_t}} +
\sum_{(s,t,u)} \lambda_s
\lambda_t \lambda_u \frac{p_{\lambda_s + \lambda_t + \lambda_u}}{p_{\lambda_s}
p_{\lambda_t} p_{\lambda_u}}\Biggr],
\end{eqnarray*}
where $\sum_{(s,t,u)}$ denotes summation over all distinct triples.

Next we compute the operators $D((1),(1);0)$ and $D(\varnothing,\varnothing;2)$.
%
\begin{eqnarray}
D\bigl((1),(1);0\bigr)p_\lambda&=& a_\delta(x)^{-1}
\sum_{w \in S_N} \varepsilon (w) x^{w \delta} \sum
_{i=1}^N (w \delta)_i
(x_i \partial_i) p_\lambda
\nonumber\\[-9pt]\label{D1,1;0} \\[-9pt]
&=&p_\lambda\sum_{s=1}^{\ell(\lambda)}
\lambda_s \biggl(\frac{1}{2} \mathop{\sum
_{r_1 + r_2 = \lambda_s}}_{r_i \ge1} \frac
{p_{r_1}p_{r_2}}{p_{\lambda_s}} +
\frac{2N-\lambda_s -1}{2}\biggr) \nonumber
\end{eqnarray}
and
\begin{eqnarray}
D(\varnothing,\varnothing;2) p_\lambda&=& a_\delta(x)^{-1}
\sum_{w \in
S_N} \varepsilon(w) x^{w \delta} \sum
_{i=1}^N (x_i
\partial_i)^2 p_\lambda
\nonumber\\[-9pt]\label{D0,0;2} \\[-9pt]
&=& p_\lambda\Biggl(\sum_{s=1}^{\ell(\lambda)}
\lambda_s^2 + \sum_{s
\neq t}
\lambda_s \lambda_t \frac{p_{\lambda_s + \lambda
_t}}{p_{\lambda_s} p_{\lambda_t}}\Biggr).\nonumber
\end{eqnarray}
We can now compute the action of $D_\theta^2$ on power sum
polynomials; see \cite{Macdonald}, Example VI.3.3(e).
%
\begin{eqnarray}\label{Dtheta^2}
D_\theta^2 p_\lambda &=&a_\lambda(x)^{-1}
\sum_{w \in S_N} \varepsilon (w) x^{w \delta}
\sum_{i \neq j} \biggl(\frac{1}{2} (w
\delta)_i (w \delta)_j + \theta(w \delta)_i(x_j \partial_j) \nonumber
\\
&&\hspace*{175pt} {} + \frac{\theta^2}{2} (x_i\partial_i) (x_j \partial_j)\biggr)p_\lambda\nonumber
\\
&=&a_\delta(x)^{-1} \sum_w\varepsilon(w) x^{w \delta} \biggl(\frac{1}{2} \biggl[\biggl(\sum
_i (w \delta)_i\biggr)^2 - \sum_i (w \delta)_i^2 \biggr]\nonumber
\\
&&\hspace*{98pt}{} +
\theta\biggl[ \biggl(\sum_i (w
\delta)_i \biggr) \biggl(\sum_i
x_i \partial_i\biggr)
- \sum_i (w \delta)_i
(x_i \partial_i)\biggr]\nonumber
\\
&&\hspace*{160pt}{}  + \frac{\theta^2}{2} \biggl[
\biggl(\sum_i x_i \partial_i
\biggr)^2 - \sum_i (x_i
\partial_i)^2\biggr]\biggr)
\\
&=& \frac{1}{2}\biggl[\biggl(\frac{N(N-1)}{2}\biggr)^2 - \frac{(N-1)N(2N-1)}{6}\biggr]\nonumber
\\
&&{}+ \theta\biggl[\frac{N(N-1)n}{2} - D\bigl((1),(1);0\bigr)\biggr] +
\frac{\theta^2}{2} \bigl[n^2 - D(\varnothing,\varnothing;2)\bigr]
p_\lambda\nonumber
\\
&=& \biggl(\frac{1}{2}\biggl[\theta^2 n^2 + \theta
n N(N-1) + \frac{N(N-1)(N-2)(3N-1)}{12}\biggr]\nonumber
\\
&&\hspace*{98pt}{} - \theta D\bigl((1),(1);0\bigr) -
\frac{\theta^2}{2}D(\varnothing,\varnothing;2)\biggr) p_\lambda.\nonumber
\end{eqnarray}
Similarly we can compute $D_\theta^3$ using the inclusion-exclusion principle,
\begin{eqnarray*}
D_\theta^3 p_\lambda &=& a_\delta(x)^{-1}
\\
&&{}\times \sum_w \varepsilon(w) x^{w
\delta}  \sum_{(i,j,k)}\biggl[ \frac{1}{6} (w
\delta)_i (w \delta)_j (w \delta)_k +
\frac{\theta}{3} (w \delta)_i (w \delta)_j
x_k \partial_k
\\
&&\hspace*{96pt}{}+ \frac{\theta^2}{3} (w\delta)_i (x_j
\partial_j) (x_k \partial_k)
+ \frac{\theta^3}{6}(x_i \partial_i) (x_j
\partial_j) (x_k \partial _k)\biggr]
\\
&=&
\biggl(c_N(3,\theta) +\frac{\theta^3}{6} \bigl[2D(\varnothing,
\varnothing;3)- 3n D(\varnothing,\varnothing;2)\bigr]
\\
&&\hspace*{4pt}{} + \frac{\theta^2}{3}\biggl[2D
\bigl((1),(2);0\bigr)- 2 n D\bigl((1),(1);0\bigr)
- \pmatrix{N\cr 2} D(\varnothing,\varnothing;2)\biggr]
\\
&&\hspace*{83pt}{} + \frac{2\theta}{3}\biggl[ D
\bigl((2),(1);0\bigr) - \pmatrix{N\cr 2} D\bigl((1),(1);0\bigr)\biggr] \biggr)
p_\lambda,
\end{eqnarray*}
where $c_N(3,\theta) = \frac{1}{6}\bigl[ {N\choose 2}^3 - \frac
{3}{4} {N\choose 2} {2N\choose 3} + 2{N\choose 2}^2\bigr]+
\frac{\theta}{2} \bigl[{N\choose 2}^2 n - \frac{1}{4}
{2N\choose 3} n\bigr] +\break \frac{\theta^2}{3} {N\choose 2} n^2 + \frac
{\theta^3 n^3}{6}$.
Unfortunately I cannot extract any natural interpretation of Markov
chains from the right-hand side. This is not so surprising since the
composite walk $P_\theta^k$ for $k \ge2$ does not correspond to some
affine transformation of the Metropolization of $P_1^k$ with respect to
the measure $\operatorname{MED}(\theta)$. Nonetheless this gives a new
Markov chain that converges to the multivariate Ewens distribution with
parameter $\theta^{-1}$, since the operators $D_\theta^r$ are
simultaneously diagonalized and the left eigenfunction corresponding to
the eigenvalue $1$ is simply the stationary distribution.

I also computed a numerical example using the symmetric reduction
function in mathematica and the SF package in maple. We take the power
sum polynomial $p_\lambda$ with $\lambda= (3,1^2)$:
\begin{eqnarray*}
D_\theta^2 p_3 p_1^2
&=& -3\theta p_2 p_1 ^3+\bigl(33\theta+35+7
\theta^2\bigr) p_3 p_1 ^2-6
\theta^2 p_4 p_1 - p_2
p_3 \theta^2,
\\
D_\theta^3 p_3 p_1^2
&=& (2/3)\theta p_1^5+\bigl(-4\theta^2-8\theta
\bigr) p_2 p_1^3
\\
&&{} +\bigl(22
\theta^2+(73/6)\theta^3+50+(307/3)\theta\bigr)
p_3 p_1^2+4\theta ^2
p_1 p_2^2
\\
&&{}+\bigl(4\theta^3-20\theta^2\bigr) p_4
p_1+\bigl(-(4/3)\theta^3-2\theta^2\bigr)
p_3 p_2+6 \theta^3 p_5,
\\
D_\theta^2 \circ D_\theta^2
p_3 p_1^2 &=& 3\theta^2
p_1^5+\bigl(-21\theta -4\theta^3-19
\theta^2\bigr) p_2 p_1^3
\\
&&{} +\bigl(505\theta^3+1225+2310\theta+ 4\theta^4+1579
\theta^2\bigr) p_3 p_1^2
+24\theta^3 p_1 p_2^2
\\
&&{} +\bigl(-41\theta^3-6\theta^4-42\theta^2\bigr)
p_4 p_1
\\
&&{} +\bigl(-6\theta^3-
\theta^4-7\theta^2\bigr) p_3
p_2+30 p_\theta^4.
\end{eqnarray*}
This example shows that $D_\theta^2$, $D_\theta^3$, $D_\theta^2
\circ D_\theta^2$ and $\operatorname{id}$ are independent operators
on~$\Lambda_N$. Notice also that $D_\theta^3 p_\lambda$ has
positivity issues: the partitions of Cayley distance~$2$ from the
starting partition $\lambda$ are always positive, whereas the ones
that differ from $\lambda$ by one transposition might become negative.
So in order to make $D_\theta^3$ into a Markov matrix, one needs to
add a sufficiently negative multiple of $D_\theta^2$. We have not
tried to compute the optimal multiple here since we are unable to glean
any nice pattern from the numerical example above; in particular,\vspace*{1pt} the
coefficients cannot be made into simple powers of $\theta$. Observe
that $D_\theta^2 \circ D_\theta^2 p_\lambda$ also has positivity
problem, but it is much easier to fix, since one can simply add a
multiple of the identity to~$D_\theta^2$ as in the case treated by
\cite{DiaconisHanlon}.

\section{Compositions of \texorpdfstring{$D_\theta^2$}{$D_theta^2$} for different \texorpdfstring{$\theta$}{$theta$} values}
In general the eigenvalues and eigenfunctions of a Markov chain can be
highly intractable, due to the need to solve for high degree
polynomials. For instance, the Metropolis chain based on $3$-cycle
shuffle on $\mathcal{P}_n$ already requires taking square roots for $n=4$:
\begin{eqnarray*}
M_4^{(3)}(\theta) &=& \pmatrix{ 0& 0& 0& 1& 0
\cr
0& 1/8&
3/4& 0& 1/8
\cr
0& 1& 0& 0& 0
\cr
\theta^2& 0& 0&1 -
\theta^2& 0
\cr
0& \theta^2& 0& 0& 1 -
\theta^2}.
\end{eqnarray*}
The eigenvalues are
\begin{eqnarray*}
&& \bigl\{1,1,-\theta^2,\tfrac{1}{16} \bigl(1-8
\theta^2-\sqrt {193-208 \theta^2+64 \theta^4}
\bigr),
\\
&&\hspace*{54pt}\tfrac{1}{16} \bigl(1-8 \theta^2+\sqrt{193-208
\theta^2+64 \theta^4} \bigr) \bigr\}.
\end{eqnarray*}
The following result was discovered in numerical experiments:

\begin{prop} \label{Laurent}
For any Laurent polynomial $p$ in $m$ variables,\break $p(D_{\theta_1}^2,
\ldots, D_{\theta_m}^2)$ gives\vspace*{1pt} rise to a Markov chain on the set of
partitions $\mathcal{P}_n$, with eigenvalues, and left and right
eigenvectors given by rational functions of $\theta_1, \ldots, \theta
_m$. In particular, the stationary distribution is given by rational
functions of $\theta_i$'s also.
\end{prop}

\begin{pf}
When expressed in the monomial symmetric basis, $D_\theta^2$ is
unipotent (upper triangular with $1$'s on the diagonal), with respect
to any total ordering on $\mathcal{P}_n$ compatible with the natural
partial ordering $\preceq$, whereby $\lambda\preceq\mu$ if $\lambda
_1 + \cdots+ \lambda_r \le\mu_1 + \cdots+ \mu_r$ for all $r$; see
\cite{Macdonald}, page 317, equation (3.7). Thus fixing this
total-ordering, any Laurent polynomial of $D_{\theta_i}$'s is clearly
still unipotent. The eigenvalues are simply the diagonal entries, and
the eigenvectors can be computed using simple row reduction, which also
results in rational components.
\end{pf}

The above result is clearly also true for $D_\theta^k$ in general and
even Macdonald operators. Thus\vspace*{1pt} in principle, one can compute the
$\mathcal{L}^2$ mixing time of Markov chain of the form $D_{\theta
_1}^2 D_{\theta_2}^2$, whose stationary distribution can be quite
complicated: for $n=3$, the stationary probabilities are
\begin{eqnarray*}
&&\biggl\{\frac{\theta_2  (\theta_1 (3-4 \theta_2)+\theta_1^2
(-1+\theta_2)+\theta_2 )}{8+\theta_2+\theta_1^2 (-1+\theta
_2) \theta_2-\theta_1  (-1+9 \theta_2+\theta_2^2
)},
\\
&&\hspace*{6.5pt} \frac{3  (-(-3+\theta_2) \theta_2+\theta_1  (1-4
\theta_2+\theta_2^2 ) )}{8+\theta_2+\theta_1^2
(-1+\theta_2) \theta_2-\theta_1  (-1+9 \theta_2+\theta
_2^2 )},
\\
&&\hspace*{8.5pt} \frac{-2 \theta_1+2 (-2+\theta_2)^2}{8+\theta_2+\theta_1^2
(-1+\theta_2) \theta_2-\theta_1  (-1+9 \theta_2+\theta
_2^2 )}\biggr\}.
\end{eqnarray*}

\section{Extensions to other root systems}

The appropriate generalization of the Laplace--Beltrami operator to
root systems other than $A_N$ is given by the Heckman--Opdam operator
(see \cite{HeckmanOpdamI} and \cite{Beerends}),
\begin{eqnarray*}
L_N(\kappa,R) &=& \Delta+ \kappa V_N:= \sum
_{i=1}^N \partial_{t_i}^2 +
\sum_{\alpha\in R_+} \kappa_\alpha\coth(\alpha/2)
\partial _\alpha,
\end{eqnarray*}
where $R$ denotes an arbitrary root system, $R_+$ a designated set of
positive roots and $\kappa_\alpha$ is called a multiplication
function, invariant under the action of the Weyl group on $R_+$. For
more on root systems and Weyl groups, consult the first 3 chapters of~\cite{GoodmanWallach} as well as Chapters~19 and 20 of \cite{Bump}.

Fascinated by the success of the $A_N$ root system, Diaconis raised the
following:

\begin{question}
Are there other root systems beside those of type $A_N$ whose
associated Heckman--Opdam operators give rise to nontrivial Markov
chains with algebraically tractible spectral decomposition?
\end{question}

We study root system $D_N$ in detail here; $B_N$ and $C_N$ are similar.
These come from the irreducible decomposition of the adjoint
representation of the maximal torus in the\vspace*{1pt} compact Lie groups
$\mathit{SO}(2N,\mathbb{R})$. The positive roots can be chosen as the set $\{
x_i x_j^{-1}, x_i x_j\dvtx  1\le i < j \le N\}$ on the maximal torus; in the
associated Cartan subalgebra (the Lie subalgebra corresponding to the
maximal torus), they become $\{t_i - t_j, t_i + t_j\dvtx  1 \le i < j \le N\}
$. The appropriate analogue of the power sum polynomials appears to be
the following power sum symmetric Laurent polynomials:
\[
p_a = \sum_{i=1}^N \cosh(a
t_i) = \sum_{i=1}^N
\bigl[x_i^a + x_i^{-a}\bigr]/2,
\]
where $x_i = e^{t_i}$. And as in the case of $A_N$, $p_\lambda= \prod_{i=1}^{\ell(\lambda)} p_{\lambda_i}$. By direct computation we have
\begin{eqnarray*}
\Delta p_\lambda&=& p_\lambda\Biggl\{\sum
_{i=1}^{\ell(\lambda)} \lambda _i^2 +
\sum_{1 \le i < j \le\ell(\lambda)} \lambda_i
\lambda_j\biggl[ \frac{p_{\lambda_i + \lambda_j}}{p_{\lambda_i} p_{\lambda_j}} - \frac{p_{\lambda_i - \lambda_j}}{p_{\lambda_i}p_{\lambda_j}}\biggr]
\Biggr\},
\\[-2pt]
\sum_{\theta\in R_+} \coth(\theta/2) \partial_\alpha
p_a &=& 2a \sum_{i \neq j} \sum
_{\ell=0}^{a-1} \cosh(\ell t_i) \cosh
\bigl((a-\ell ) t_j\bigr)
\\[-2pt]
&=&\bigl(2an -a^2 -a\bigr)p_a + 2a \sum
_{\ell=1}^{a-1} p_\ell p_{a - \ell} - a
\sum_{\ell=1}^{a-1} p_{a -2 \ell},
\\[-2pt]
\sum_{\ell=1}^{a-1} p_{a - 2 \ell} &=&
\cases{
\displaystyle 2 \sum_{\ell=1}^{(a-1)/2} p_{a - 2\ell}, &\quad if $a$ is odd,
\vspace*{5pt}\cr
\displaystyle N+ 2 \sum _{\ell=1}^{(a-2)/2} p_{a - 2\ell}, &\quad if $a$ is
even}
\\[-2pt]
&=& 2 \sum_{\ell=1}^{\lfloor {a}/{2} \rfloor} p_{a -2\ell},
\end{eqnarray*}
if we define $p_0:= N/2$.

Therefore for $n = |\lambda| = \sum_i \lambda_i$,
\begin{eqnarray*}
&& \sum_{\alpha\in R_+} \coth(\alpha/2) \partial_\alpha
p_\lambda
\\
&&\qquad = \Biggl((2N-1)n - \sum_{i=1}^{\ell(\lambda)}
\lambda_i^2\Biggr) p_\lambda+\sum
_{i=1}^{\ell(\lambda)} \frac{p_\lambda}{p_{\lambda_i}} 2 \lambda
_i \Biggl[ \sum_{\ell=1}^{\lambda_i -1}
p_\ell p_{\lambda_i - \ell} - \sum_{\ell=1}^{\lfloor\lambda_i /2 \rfloor}
p_{\lambda_i - 2 \ell}\Biggr].
\end{eqnarray*}

Restricting to partitions of the top grading, $n$, clearly the
transition coefficients are affine transformation of those in the $A_N$
case, and hence nothing new is obtained this way. There are several
pathological features regarding the action of $L_N(D, \kappa):= \Delta
+ \sum_{\alpha\in R_+} \kappa_\alpha\coth(\alpha/2) \partial
_\alpha$ on the power sum analogues of symmetric Laurent polynomials
(see the toy example below):
\begin{longlist}[(2)]
\item[(1)] there are no easy ways to make the entries all positive;
\item[(2)] the row sums are not the same.
\end{longlist}
Thus it remains difficult to interpret the full transition matrix as a
Markov kernel. For root system $D_N$, there is only one Weyl orbit,
hence $\kappa_\alpha\equiv\kappa$. The Heckman--Opdam functions
have rational transition coefficients to this power sum Laurent basis,
as illustrated by the following numerical example ($\mathcal{P}_k$
denotes the set of partitions of $k$):
{\fontsize{10pt}{13pt}\selectfont{\begin{eqnarray*}
\hspace*{-5pt}&& M_\kappa(N)|_{\mathcal{P}_3 \cup\mathcal{P}_1}
\\
\hspace*{-5pt}&&\qquad:=
\pmatrix{ 9 + \kappa(-12 + 6N)& 12 \kappa& 0& -6 \kappa
\cr
2& 5 + \kappa(-8 + 6 N)& 4 \kappa& -2 - 2 \kappa N
\cr
0& 3& 3 + \kappa(-6 + 6 N)& -3
\cr
0& 0& 0& 1 +
\kappa(-2 + 2 N)},
\end{eqnarray*}}}%
where the columns and rows are indexed by $(3),(21),(1^3),(1)$.
The eigenvalues are very clean,
\begin{eqnarray*}
& \displaystyle 3 + 6 \kappa(-2 + N),\qquad
9 + 6 \kappa(-1 + N),&
\\
&\displaystyle 5 +  \kappa(-8 + 6 N), \qquad
1 + 2 \kappa(-1 + N).&
\end{eqnarray*}
The left eigenvectors are rational functions of $\kappa$, which we
display as rows of the following matrix:\vspace*{5pt}
{\fontsize{10.1pt}{13pt}\selectfont{\begin{eqnarray*}
\pmatrix{ \displaystyle \frac{-5 +\theta+ 2 N}{3 (-1 + N)}& \displaystyle-\frac{-5+\theta+ 2 N}{-1 + N}&
\displaystyle\frac{2 (-5 +\theta+ 2 N)}{3 (-1 + N)}& 1
\vspace*{7pt}\cr
0& 0& 0& 1
\vspace*{7pt}\cr
\displaystyle -
\frac{2\theta(-1 + 2\theta+ N)}{3 (1 + 2\theta+ N)}& \displaystyle\frac {-2 (-1 + 2\theta+ N)}{1 + 2\theta+ N}& \displaystyle-
\frac{4 (-1 + 2\theta+ N)}{
3\theta(1 + 2\theta+ N)}& 1
\vspace*{7pt}\cr
\displaystyle\frac{\theta(-3 + 2\theta+ 2 N)}{\theta+ 2\theta^2 - 2 N + 2\theta N}& \displaystyle-
\frac{ 2 (-1 +\theta) (-3 +
2\theta+ 2 N)}{\theta+ 2\theta^2 - 2 N + 2\theta N}& \displaystyle\frac{12 - 8\theta- 8 N}{\theta+ 2\theta^2 - 2 N + 2\theta N}& 1}.
\end{eqnarray*}}}\vspace*{10pt}%
Here $\theta= \kappa^{-1}$ corresponds to the parameter in the $A_N$ case.

We have also tried to adjust the diagonal entries to make the row sum
equal to~$1$; the resulting matrix however does not have rational
eigenvalues in the entries.
\end{appendix}

\section*{Acknowledgments}
I thank Persi Diaconis for introducing me to this problem as well as
providing guidance throughout the writing. I also thank Calvin
(Yuncheng) Lin for suggesting that the Lie algebra formed by $D_\alpha
$ and natural Lie bracket is solvable and hence can be simultaneously
uni-upper-triangularized by a result of Serre. Several of the results
in the \hyperref[app]{Appendix} are computed and verified using Mathematica 6.0 and
Maple 13 together with the package SF developed by John Stembridge.


%

\printaddresses

\begin{thebibliography}{22}

\bibitem{EwensPoisson}
%
\begin{barticle}[mr]
\bauthor{\bsnm{Arratia},~\bfnm{Richard}\binits{R.}},
\bauthor{\bsnm{Barbour},~\bfnm{A.~D.}\binits{A.~D.}} \AND
\bauthor{\bsnm{Tavar{\'e}},~\bfnm{Simon}\binits{S.}}
(\byear{1992}).
\btitle{Poisson process approximations for the {E}wens sampling formula}.
\bjournal{Ann. Appl. Probab.}
\bvolume{2}
\bpages{519--535}.
\bid{issn={1050-5164}, mr={1177897}}
\end{barticle}
%
\bptok{imsref}%
\endbibitem

\bibitem{dovetail}
%
\begin{barticle}[mr]
\bauthor{\bsnm{Bayer},~\bfnm{Dave}\binits{D.}} \AND
\bauthor{\bsnm{Diaconis},~\bfnm{Persi}\binits{P.}}
(\byear{1992}).
\btitle{Trailing the dovetail shuffle to its lair}.
\bjournal{Ann. Appl. Probab.}
\bvolume{2}
\bpages{294--313}.
\bid{issn={1050-5164}, mr={1161056}}
\end{barticle}
%
\bptok{imsref}%
\endbibitem

\bibitem{Beerends}
%
\begin{barticle}[mr]
\bauthor{\bsnm{Beerends},~\bfnm{R.~J.}\binits{R.~J.}}
(\byear{1991}).
\btitle{Chebyshev polynomials in several variables and the radial part
of the {L}aplace--{B}eltrami operator}.
\bjournal{Trans. Amer. Math. Soc.}
\bvolume{328}
\bpages{779--814}.
\bid{doi={10.2307/2001804}, issn={0002-9947}, mr={1019520}}
\end{barticle}
%
\bptok{imsref}%
\endbibitem

\bibitem{Bump}
%
\begin{bbook}[mr]
\bauthor{\bsnm{Bump},~\bfnm{Daniel}\binits{D.}}
(\byear{2004}).
\btitle{Lie Groups}.
\bseries{Graduate Texts in Mathematics}
\bvolume{225}.
\bpublisher{Springer},
\blocation{New York}.
\bid{doi={10.1007/978-1-4757-4094-3}, mr={2062813}}
\end{bbook}
%
\bptok{imsref}%
\endbibitem

\bibitem{PD}
%
\begin{bbook}[mr]
\bauthor{\bsnm{Diaconis},~\bfnm{Persi}\binits{P.}}
(\byear{1988}).
\btitle{Group Representations in Probability and Statistics}.
\bseries{Institute of Mathematical Statistics Lecture Notes---Monograph Series}
\bvolume{11}.
\bpublisher{IMS},
\blocation{Hayward, CA}.
\bid{mr={0964069}}
\end{bbook}
%
\bptok{imsref}%
\endbibitem

\bibitem{Aux}
\begin{barticle}[mr]
\bauthor{\bsnm{Diaconis},~\bfnm{Persi}\binits{P.}} \AND
\bauthor{\bsnm{Ram},~\bfnm{Arun}\binits{A.}}
(\byear{2012}).
\btitle{A probabilistic interpretation of the {M}acdonald polynomials}.
\bjournal{Ann. Probab.}
\bvolume{40}
\bpages{1861--1896}.
\bid{doi={10.1214/11-AOP674}, issn={0091-1798}, mr={3025704}}
\end{barticle}
\bptok{imsref}%
\endbibitem

\bibitem{DiaconisShahshahani}
%
\begin{barticle}[mr]
\bauthor{\bsnm{Diaconis},~\bfnm{Persi}\binits{P.}} \AND
\bauthor{\bsnm{Shahshahani},~\bfnm{Mehrdad}\binits{M.}}
(\byear{1981}).
\btitle{Generating a random permutation with random transpositions}.
\bjournal{Z. Wahrsch. Verw. Gebiete}
\bvolume{57}
\bpages{159--179}.
\bid{doi={10.1007/BF00535487}, issn={0044-3719}, mr={0626813}}
\end{barticle}
%
\bptok{imsref}%
\endbibitem

\bibitem{DummitFoote}
%
\begin{bbook}[mr]
\bauthor{\bsnm{Dummit},~\bfnm{David~S.}\binits{D.~S.}} \AND
\bauthor{\bsnm{Foote},~\bfnm{Richard~M.}\binits{R.~M.}}
(\byear{1991}).
\btitle{Abstract Algebra}.
\bpublisher{Prentice Hall},
\blocation{Englewood Cliffs, NJ}.
\bid{mr={1138725}}
\end{bbook}
%
\bptok{imsref}%
\endbibitem

\bibitem{FH}
%
\begin{bbook}[mr]
\bauthor{\bsnm{Fulton},~\bfnm{William}\binits{W.}} \AND
\bauthor{\bsnm{Harris},~\bfnm{Joe}\binits{J.}}
(\byear{1991}).
\btitle{Representation Theory: A First Course}.
\bseries{Graduate Texts in Mathematics}
\bvolume{129}.
\bpublisher{Springer},
\blocation{New York}.
\bid{doi={10.1007/978-1-4612-0979-9}, mr={1153249}}
\end{bbook}
%
\bptok{imsref}%
\endbibitem

\bibitem{GoodmanWallach}
%
\begin{bbook}[mr]
\bauthor{\bsnm{Goodman},~\bfnm{Roe}\binits{R.}} \AND
\bauthor{\bsnm{Wallach},~\bfnm{Nolan~R.}\binits{N.~R.}}
(\byear{1998}).
\btitle{Representations and Invariants of the Classical Groups}.
\bseries{Encyclopedia of Mathematics and Its Applications}
\bvolume{68}.
\bpublisher{Cambridge Univ. Press},
\blocation{Cambridge}.
\bid{mr={1606831}}
\end{bbook}
%
\bptok{imsref}%
\endbibitem

\bibitem{Hanlon}
%
\begin{barticle}[mr]
\bauthor{\bsnm{Hanlon},~\bfnm{Phil}\binits{P.}}
(\byear{1992}).
\btitle{A {M}arkov chain on the symmetric group and {J}ack symmetric functions}.
\bjournal{Discrete Math.}
\bvolume{99}
\bpages{123--140}.
\bid{doi={10.1016/0012-365X(92)90370-U}, issn={0012-365X}, mr={1158785}}
\end{barticle}
%
\bptok{imsref}%
\endbibitem

\bibitem{HeckmanOpdamI}
%
\begin{barticle}[mr]
\bauthor{\bsnm{Heckman},~\bfnm{G.~J.}\binits{G.~J.}} \AND
\bauthor{\bsnm{Opdam},~\bfnm{E.~M.}\binits{E.~M.}}
(\byear{1987}).
\btitle{Root systems and hypergeometric functions.~{I}}.
\bjournal{Compos. Math.}
\bvolume{64}
\bpages{329--352}.
\bid{issn={0010-437X}, mr={0918416}}
\end{barticle}
%
\bptok{imsref}%
\endbibitem


\bibitem{Johnson}
%
\begin{bbook}[mr]
\bauthor{\bsnm{Johnson},~\bfnm{Norman~L.}\binits{N.~L.}},
\bauthor{\bsnm{Kotz},~\bfnm{Samuel}\binits{S.}} \AND
\bauthor{\bsnm{Balakrishnan},~\bfnm{N.}\binits{N.}}
(\byear{1997}).
\btitle{Discrete Multivariate Distributions}.
\bpublisher{Wiley},
\blocation{New York}.
\bid{mr={1429617}}
\end{bbook}
%
\bptok{imsref}%
\endbibitem

\bibitem{Koike}
%
\begin{barticle}[mr]
\bauthor{\bsnm{Koike},~\bfnm{Kazuhiko}\binits{K.}}
(\byear{1993}).
\btitle{On a conjecture of {S}tanley on {J}ack symmetric functions}.
\bjournal{Discrete Math.}
\bvolume{115}
\bpages{211--216}.
\bid{doi={10.1016/0012-365X(93)90490-K}, issn={0012-365X}, mr={1217630}}
\end{barticle}
%
\bptok{imsref}%
\endbibitem

\bibitem{Koornwinder}
%
\begin{bincollection}[mr]
\bauthor{\bsnm{Koornwinder},~\bfnm{Tom~H.}\binits{T.~H.}}
(\byear{1994}).
\btitle{Special functions associated with root systems: Recent progress}.
In \bbooktitle{From Universal Morphisms to Megabytes: A {B}aayen Space Odyssey}
\bpages{391--404}.
\bpublisher{Math. Centrum, Centrum Wisk. Inform.},
\blocation{Amsterdam}.
\bid{mr={1490602}}
\end{bincollection}
%
\bptok{imsref}%
\endbibitem

\bibitem{Yuval-book}
%
\begin{bbook}[mr]
\bauthor{\bsnm{Levin},~\bfnm{David~A.}\binits{D.~A.}},
\bauthor{\bsnm{Peres},~\bfnm{Yuval}\binits{Y.}} \AND
\bauthor{\bsnm{Wilmer},~\bfnm{Elizabeth~L.}\binits{E.~L.}}
(\byear{2009}).
\btitle{Markov Chains and Mixing Times}.
\bpublisher{Amer. Math. Soc.},
\blocation{Providence, RI}.
\bid{mr={2466937}}
\end{bbook}
%
\bptok{imsref}%
\endbibitem

\bibitem{Macdonald}
%
\begin{bbook}[mr]
\bauthor{\bsnm{Macdonald},~\bfnm{I.~G.}\binits{I.~G.}}
(\byear{1995}).
\btitle{Symmetric Functions and {H}all Polynomials},
\bedition{2nd} ed.
\bseries{Oxford Mathematical Monographs}.
\bpublisher{Oxford Univ. Press},
\blocation{New York}.
\bid{mr={1354144}}
\end{bbook}
%
\bptok{imsref}%
\endbibitem

\bibitem{MacdonaldRoots}
%
\begin{barticle}[mr]
\bauthor{\bsnm{Macdonald},~\bfnm{I.~G.}\binits{I.~G.}}
(\byear{2000/2001}).
\btitle{Orthogonal polynomials associated with root systems}.
\bjournal{S\'em. Lothar. Combin.}
\bvolume{45}
\bpages{Art. B45a, 40 pp. (electronic)}.
\bid{issn={1286-4889}, mr={1817334}}
\end{barticle}
%
\bptok{imsref}%
\endbibitem

\bibitem{MorrisLreversal}
%
\begin{barticle}[mr]
\bauthor{\bsnm{Morris},~\bfnm{Ben}\binits{B.}}
(\byear{2009}).
\btitle{Improved mixing time bounds for the {T}horp shuffle and
{$L$}-reversal chain}.
\bjournal{Ann. Probab.}
\bvolume{37}
\bpages{453--477}.
\bid{doi={10.1214/08-AOP409}, issn={0091-1798}, mr={2510013}}
\end{barticle}
%
\bptok{imsref}%
\endbibitem

\bibitem{StanleyJack}
%
\begin{barticle}[mr]
\bauthor{\bsnm{Stanley},~\bfnm{Richard~P.}\binits{R.~P.}}
(\byear{1989}).
\btitle{Some combinatorial properties of {J}ack symmetric functions}.
\bjournal{Adv. Math.}
\bvolume{77}
\bpages{76--115}.
\bid{doi={10.1016/0001-8708(89)90015-7}, issn={0001-8708}, mr={1014073}}
\end{barticle}
%
\bptok{imsref}%
\endbibitem

\bibitem{DiaconisHanlon}
%
\begin{barticle}[auto]
\bauthor{\bsnm{Diaconis},~\bfnm{Persi}\binits{P.}} \AND
\bauthor{\bsnm{Hanlon},~\bfnm{Phil}\binits{P.}}
(\byear{1992}).
\btitle{Eigen-analysis for some examples of the {M}etropolis algorithm}.
\bjournal{Contemp. Math.}
\bvolume{138}
\bpages{99--117}.
\end{barticle}
%
\bptok{imsref}%
\endbibitem

\end{thebibliography}
\end{document}